\documentclass{article}

\usepackage{comment}
\usepackage[12pt]{extsizes}
\usepackage{cmap}
\usepackage[utf8]{inputenc}
\usepackage[T2A]{fontenc}
\usepackage{setspace}
\usepackage[english]{babel}
\usepackage{amsmath}
\usepackage[
top    = 2.50cm,
bottom = 2.75cm,
left   = 2.50cm,
right  = 2.50cm]{geometry}
\usepackage{amsthm}
\usepackage{amssymb}
\usepackage{amsfonts}
\usepackage{tensor}
\usepackage{enumerate}
\usepackage[colorlinks=true,linkcolor=magenta,citecolor=blue]{hyperref}

\newcommand{\RN}{\mathbb{R}} %real numbers
\newcommand{\CN}{\mathbb{C}} %complex numbers
\newcommand{\ZN}{\mathbb{Z}} %Zahlen

\newcommand{\eps}{\ensuremath\varepsilon}

\newcommand{\ovl}{\overline}

\newcommand{\p}{\partial}
\renewcommand{\l}{\lambda}

\newcommand{\mc}{\mathcal}
\newcommand{\mb}{\mathbf}
\newcommand{\B}{\mathbb}

\newcommand{\Bun}{\operatorname{Bun}}

\newcommand{\abs}[1]{\left\lvert #1 \right\rvert}
\newcommand{\Abs}[1]{\left\Vert #1\right\Vert}
\newcommand{\Tr}{\operatorname{Tr}}
\newcommand{\ex}[1]{\left\langle #1 \right\rangle}
\newcommand{\bx}{\mathbf{x}}
\newcommand{\by}{\mathbf{y}}
\newcommand{\bz}{\mathbf{z}}
\newcommand{\bw}{\mathbf{w}}
\newcommand{\bu}{\mathbf{u}}
\renewcommand{\bar}{\overline}

\renewcommand{\a}{\alpha}
\renewcommand{\b}{\beta}
\newcommand{\g}{\gamma}
\renewcommand{\l}{\lambda}

\begin{document}
\newtheorem{thr}{Theorem}[section]
\newtheorem*{thr*}{Theorem}
\newtheorem{lem}[thr]{Lemma}
\newtheorem*{lem*}{Lemma}
\newtheorem{cor}[thr]{Corollary}%corollary
\newtheorem*{cor*}{Corollary}
\newtheorem{prop}[thr]{Proposition}
\newtheorem*{prop*}{Proposition}
\newtheorem{stat}[thr]{Statement}
\newtheorem*{stat*}{Statement}

\theoremstyle{definition}
\newtheorem{defn}[thr]{Definition}
\theoremstyle{remark}
\newtheorem{rem}[thr]{Remark}
\newtheorem*{rem*}{Remark}
\newtheorem{example}[thr]{Example}

\title{Multiplication Kernels for the Analytic Langlands Program in Genus Zero}
\author{Daniil Klyuev \and Sanjay Raman}
\date{\today}

\maketitle

\abstract{We provide an explicit proof of a recent result of Gaiotto \cite{Gaiotto2021} which gives an explicit formula for a so-called ``multiplication kernel'' $K_3(x, y, z; t)$ intertwining the action of Hecke operators and Gaudin operators in three sets of variables. This function $K_3$ arises naturally in the context of the analytic formulation of the geometric Langlands program in the genus-zero case \cite{EFK1,  EFK3, EFK2}. We also discuss how the kernel $K_3$ relates to other objects typically considered in the analytic Langlands program.}

\tableofcontents

\onehalfspacing

\pagebreak
\section{Introduction} 

\begin{comment}
The geometric Langlands program is a web of far-reaching conjectures in algebraic geometry extending the original number-theoretic Langlands program conjectured by Langlands in 1967. In this setup, let $G$ be a reductive algebraic group over a field $F$, and let $\tensor[^L]{G}{}$ be the Langlands dual of $G$. Now, consider a curve $X$ over $F$. A prominent object of interest is then the category of $D$-modules on the moduli stack of $G$-bundles over $X$ (denoted $\mathrm{Bun}_G(X)$). The original geometric Langlands program, due to Bellinson and Drinfeld \cite{1885368}, posits a correspondence of sorts between this category of $D$-modules and the category of quasicoherent sheaves on the moduli stack of $\tensor[^L]{G}{}$-local systems on $X$. Here, a $\tensor[^L]{G}{}$-local system is given by a $\tensor[^L]{G}{}$-bundle over $X$ with a connection. 
\end{comment}

Recently, Etingof, Frenkel, and Kazhdan devised a version of the Langlands correspondence in terms of functional analysis \cite{EFK1, EFK3, EFK2}. Taking a smooth irreducible projective curve $X$ over $\B{C}$, they formulated the geometric Langlands duality as a correspondence between the joint spectrum of global differential operators (\textit{quantum Hitchin Hamiltonians}) acting on a Hilbert space $\mathcal{H}$ of half-densities on $\mathrm{Bun}_G(X)$ and certain points in the variety $\mathrm{Op}_{\tensor[^L]{G}{}}(X)$ of $\tensor[^L]{G}{}$-opers over $X$ (more precisely, the $\tensor[^L]{G}{}$-opers with \textit{real monodromy}) \cite{EFK1}. To generalize the problem to local fields $F$, integral \textit{Hecke operators} were introduced in \cite{Braverman05}; these operators strongly commute with the quantum Hitchin Hamiltonians over $\mathbb{C}$ and are defined over arbitrary local fields $F$. However, the joint spectrum of the Hecke operators is still not well understood over arbitrary local fields. 

To get a handle on the analytic Langlands program, it is thus important to investigate the joint spectrum of the Hecke operators and the quantum Hitchin Hamiltonians. Gaiotto recently made an intriguing claim in this direction in the genus-zero case with $G = PGL_2$ \cite{Gaiotto2021}. In particular, taking $F = \mathbb{C}$, $X = \mathbb{P}^1$, and $G = PGL_2$, Gaiotto described a kernel $K_3(x, y, z; t)$ in three sets of variables intertwining the action of the quantum Hitchin Hamiltonians $G_i$ in genus zero (\textit{Gaudin operators}):
\[ G_i^x K_3 = G_i^y K_3 = G_i^z K_3. \]
The precise definitions of the Gaudin operators will be given in Section \ref{sec:prelim}. Gaiotto's result is stated below: 
\begin{thr} \label{thm:main}
Define a matrix $A_{ij}$ with entries as follows: 
\[ A_{ij}=\frac{(x_i-x_j)(y_i-y_j)(z_i-z_j)}{t_i-t_j}, \ 0 \leq i \neq j \leq m \]
and $A_{ii} = 0$. 
Then
\[ K_3(x, y, z; t) = \frac{1}{\abs{\det A}} \]
satisfies the requisite intertwining property for the Gaudin operators. 
\end{thr}

Gaiotto gave an argument for his formula using physical insight. In this paper, we furnish a full mathematical proof of Theorem \ref{thm:main} (and a generalization thereof) and discuss its implications. The organization of this paper is as follows. In Section \ref{sec:prelim}, we introduce in further detail the relevant mathematical preliminaries needed to formulate Gaiotto's result precisely. In Section \ref{sec:proof}, we produce an explicit proof of Gaiotto's result, inspired in some part by the physical origin of the formula. We also prove a suitable generalization of Gaiotto's formula. In Section \ref{sec:discussion}, we discuss the implications of our result and consider further avenues for exploration. 

\section{Preliminaries}
\label{sec:prelim}

\subsection{Quantum Hitchin Hamiltonians and the Gaudin System}

We will now review the relevant ingredients from the analytic Langlands program. Our discussion follows the foundational work in \cite{EFK1, EFK3, EFK2}.

The starting point consists of a reductive algebraic group $G$ and an smooth irreducible projective curve $X$ over a local field $F$ with a specified zero-dimensional subvariety $S \subset X$. Let $\mathrm{Bun}_G(X, S)$ be the moduli stack of $G$-bundles over $X$ with a specified reduction to a Borel subgroup $B \subset G$ at the subvariety $S \subset X$, and let $\Bun_G^{\circ}(X, S)$ be the substack of such \textit{stable} $G$-bundles (defined precisely in Definition 2.1 of \cite{EFK3}). 

If $\abs{S}$ is sufficiently large, then $\Bun_G^{\circ}(X, S)$ is open and dense in $\mathrm{Bun}_G(X, S)$. If we take $G$ to be simple and simply-connected, and if $X$ is a curve of genus $g$, then $\Bun_G^\circ(X, S)$ can be viewed as a smooth quasiprojective variety of dimension $3g-3 + \abs{S}$ \cite{EFK3}. We then define $V_G(X, S)$ to be the space of smooth, compactly supported sections of the line bundle of half-densities $\Omega^{1/2}_{\mathrm{Bun}}$ over $\mathrm{Bun}^\circ_G(X, S)$. $V_G(X, S)$ admits an inner product: 
\[ \ex{v, w} := \int_{\mathrm{Bun}_G^\circ(X, S)} v \cdot \bar{w} . \]
We then define $\mathcal{H}_G(X, S)$ to be the Hilbert space completion of $V_G(X, S)$ with respect to $\ex{\cdot, \cdot}$.

Now, take $F = \mathbb{C}$. To make contact with the geometric Langlands program, we wish to analyze the joint spectrum of a commutative algebra $\mathcal{A}$ of \textit{global} differential operators acting on $\Omega^{1/2}_{\mathrm{Bun}}$. Specifically, take $K_{\mathrm{Bun}}$ to be the canonical line bundle on $\mathrm{Bun}^\circ_G(X, S)$. Now, it is known that $K_{\mathrm{Bun}}$ has a unique square root $K^{1/2}_{\mathrm{Bun}}$ up to isomorphism. The commutative algebra $D_G$ of \textit{global} holomorphic differential operators on $K^{1/2}_{\mathrm{Bun}}$ is generated by the \textit{quantum Hitchin Hamiltonians}, and their properties were first described in \cite{Teschner2017}. We may just as well consider the antiholomorphic operators $\bar{D_G}$ acting on the conjugate $\bar{K}^{1/2}_{\mathrm{Bun}}$ of $K^{1/2}_{\mathrm{Bun}}$. Following \cite{EFK1}, we define the commutative algebra $\mathcal{A}$ as follows: 
\[ \mathcal{A} = D_G \otimes_{\mathbb{C}} \bar{D_G}. \]
To understand the joint spectrum of $\mathcal{A}$, we wish to study the action of holomorphic differential operators on half-densities in $\Omega^{1/2}_{\mathrm{Bun}} = K^{1/2}_{\mathrm{Bun}} \otimes \bar{K}^{1/2}_{\mathrm{Bun}}$.
To establish (over $F = \mathbb{C}$) the analytic Langlands correspondence, it is thus essential to understand the joint spectra of the quantum Hitchin Hamiltonians. For general $G$ and $X$, this is quite a difficult problem; however, some progress has been made for the simplest non-abelian $G$. In \cite{EFK3, Gaiotto2021}, the specific case of $G = PGL_2$ over $X = \mathbb{P}^1$ is considered, and this is the case we shall focus on as well. In genus zero, we require $\abs{S} \geq 3$ parabolic points for the moduli stack $\mathrm{Bun}_G^\circ(X, S)$ of stable $G$-bundles to be open and dense in $\mathrm{Bun}_G(X, S)$, and for $\abs{S} = 3$, this moduli stack consists of a single point. Thus, we henceforth take
\[ S = \{t_0, \cdots, t_{m+1}\} \] 
where $m \geq 2$. Specifying a Borel reduction of $G$ at $S$ then amounts to giving each point $t_i$ a \textit{parabolic structure}, which is simply a choice of direction $y_i \in \mathbb{P}^1$. In fact, $\mathrm{Bun}_G^\circ(X, S)$ is parametrized precisely by the $y_i$. We have the following birational equivalence \cite{EFK3}:
\[ \mathrm{Bun}_G^\circ(X, S) \  \mathop{\longrightarrow}^{\sim} \ \mathbb{P}^{m-1}.\]
The fact that $\mathrm{Bun}_G^\circ(X, S)$ has dimension $(m-1)$ can intuitively be seen by the fact that we may translate by $PGL_2$ to fix three of the parabolic structures at the marked points. 

Given a choice of $y_i$, the (holomorphic) quantum Hitchin Hamiltonians reduce to the so-called \textit{Gaudin operators} \cite{Teschner2017}:
\[ G_i = \sum_{j \neq i} \frac{1}{t_j - t_i} \left[ -(y_i - y_j)^2 \p_i \p_j + (y_i - y_j)(\p_i - \p_j) + \frac{1}{2} \right],   \]
where $\p_i \equiv \p/\p y_i$. It is easy to verify that the algebra of Gaudin operators is indeed commutative: $[G_i, G_j] = 0$. Our problem is then to understand the joint spectrum of the Gaudin system. 

\subsection{Hecke Operators}

The original number-theoretic Langlands program for curves $X$ over finite fields $\mathbb{F}_q$ studies the joint spectrum of commuting Hecke operators. The notion of Hecke operators have a corresponding analog in the analytic Langlands program, where $X$ is now a curve over some local field $F$. When $F = \mathbb{C}$ or $F = \mathbb{R}$, the Hecke operators strongly commute with the quantum Hitchin Hamiltonians, so the joint spectrum of the Hecke operators can also be seen via Langlands duality to correspond to the $\tensor[^L]{G}{}$-opers with real monodromy. Understanding the spectra of the Hecke operators therefore gives us a grip on the geometric Langlands program for non-archimedean local fields $F$, where we no longer have access to the quantum Hitchin Hamiltonians. 

A full definition of Hecke operators is given in \cite{EFK2}, but we are only concerned with the case where $X = \mathbb{P}^1$ and $G = PGL_2$, so we will restrict our attention to the case of rank-two vector bundles over $X$ with marked points $t_0, \cdots, t_{m+1}$ with parabolic structures $y_0, \cdots, y_{m+1}$. The Hecke operators form a commuting set of compact operators $H_t$ for each point $t \in X$. Taking as usual $t_{m+1} = \infty$ and further setting $y_{m+1} = \infty$, there is an explicit formula for the action of the Hecke operators (Proposition 3.9 of \cite{EFK3}): 
\[ H_t \psi(y_0, \cdots, y_m) = \Abs{ \prod_{i = 0}^m (t - t_i) }^{1/2} \int_{F} \frac{\Abs{ds}}{\prod_{i = 0}^m \Abs{s - y_i}} \psi \left(\frac{t_0 - t}{s - y_0}, \cdots, \frac{t_m - t}{s - y_m}  \right) . \]
Over $F = \mathbb{C}$, the Hecke operators strongly commute with the Gaudin system, and the joint spectrum of this commutative algebra of operators is simple \cite{EFK3}. 

\subsection{The Gaiotto Kernel}

It was Gaiotto and Witten \cite{https://doi.org/10.48550/arxiv.2107.01732} who first considered the symmetric \textit{multiplication kernel} $K_3(x, y, z; t)$ in three sets of variables intertwining the action of the Gaudin and Hecke operators over $\mathbb{C}$. More precisely, if we let $x_0, \cdots, x_m, y_0, \cdots, y_m, z_0, \cdots, z_m$ be three copies of variables, they were interested in a function $K_3(x, y, z; t)$ symmetric in $x, y, z$ satisfying
\[ G_i^x K_3(x, y, z; t) = G_i^y K_3(x, y, z; t) = G_i^z K_3(x, y, z; t); \]
\[ H_{t'}^x K_3(x, y, z; t) = H_{t'}^y K_3(x, y, z; t) = H_{t'}^z K_3(x, y, z; t). \]
Here $G_i^u$ and $H_{t'}^u$ denote the Gaudin and Hecke operators, respectively, with respect to the variables $u_0, \cdots, u_m$ for $u = x, y, z$.

The object $K_3$ is of interest to us for several reasons. First of all, it furnishes a ``multiplication kernel'' in the language of \cite{Kontsevich_2021}. That is, it defines a commutative associative product $*$ on $\mathcal{H}$. Specifically, if we define
\[ (f * g)(x) = \int_{\mathrm{Bun}_G^\circ(X, S) \times \mathrm{Bun}_G^\circ(X, S)}  K_3(x, y, z) f(y) g(z) \, \Abs{dy} \Abs{dz} , \]
then the product $*$ is commutative (by virtue of the symmetry of $K_3$ in $x, y, z$) and associative (as shown in Section 4 of \cite{Gaiotto2021}).

Further, since $K_3$ intertwines the action of the Gaudin and Hecke operators, it may be shown that integral \textit{Gaiotto operators} $K_x$ defined by integrating over the kernel $K_3$ (for arbitrary $x$) commute with the Gaudin and Hecke operators. Explicitly, for $f \in \mathcal{H}$, we find
\[ (K_x \circ G_i) f = (G_i \circ K_x) f ; \]
\[ (K_x \circ H_t) f = (H_t \circ K_x) f ; \]
Thus, the Gaiotto operators $K_x$ diagonalize in the joint eigenbasis of the Gaudin operators. An explicit computation of $K_3$ would thus yield information about the joint spectrum of the Gaudin and Hecke operators. 

In his recent work \cite{Gaiotto2021}, Gaiotto produced a simple explicit expression for the kernel $K_3$ over $F = \mathbb{C}$. He defined the matrix $A_{ij}$ with entries as follows: 
\[ A_{ij}=\frac{(x_i-x_j)(y_i-y_j)(z_i-z_j)}{t_i-t_j}, \ 0 \leq i \neq j \leq m \]
and $A_{ii} = 0$. Using arguments from conformal field theory, he gave the following formula, stated above in Theorem \ref{thm:main}:
\[ K_3(x, y, z; t) = \frac{1}{\abs{\det A}}. \]
That the Gaiotto operators commute with the Hecke operators follows from the fact that they commute with the Gaudin operators, since the joint spectrum is simple \cite{EFK3}. Gaiotto provided an explicit proof of his formula for four marked points $t_i$ and a numerical verification for five points. In Section \ref{sec:proof}, we will provide a complete proof of Theorem \ref{thm:main}. We will also provide a direct proof of the analogous intertwining formula for Hecke operators in a following paper \cite{klyuev22}.

\section{Proof of Gaiotto's Formula}
\label{sec:proof}

\subsection{Outline of Proof}

In this section, we will prove Theorem \ref{thm:main}. Inspired by the physical argument, we write $K_3$ as a formal Gaussian integral and consider the action of the holomorphic Gaudin operators by differentiating inside the integral. Care must be taken to ensure that the differentiation under the integral sign is allowed and that the representation of $K_3$ as a Gaussian integral is valid. The resulting integral can be evaluated using Wick's theorem. When the dust clears, the resulting expression is symmetric in $x, y, z$. 

\subsection{Main Proof}
\label{subsec:mainproof}

To prove Gaiotto's formula, we note that the determinant formula for $K_3$ arises in Gaiotto's chiral algebra construction as the result of a Gaussian integral over formal parameters $\lambda_i$ \cite{Gaiotto2021}. To this end, we will perform some formal calculations involving Gaussian integrals and relate these to what we wish to prove. To begin, notice that we may write
\[ K_3(x, y, z; t) = \frac{1}{\abs{\det A}} = \frac{1}{\sqrt{\det A}} \cdot \frac{1}{\sqrt{\det \bar{A}}}. \]
The holomorphic Gaudin operators care only about the determinant of the \textit{holomorphic} matrix $A$, so it suffices to represent $(\det A)^{-1/2}$ as a Gaussian integral and consider the action of the Gaudin operator on this piece. To begin, we quote two well-known elementary results for real matrices:

\begin{lem}[Gaussian Integration] \label{lem:gaussian}
Let $A$ be a real, symmetric positive-definite matrix. We have (for suitable normalization of the Lebesgue measure)
$$
\int_{\Bbb R^n} e^{-\frac{1}{2}(\bold u,A\bold u)}d^n \bold u=(\det A)^{-1/2}. 
$$
\end{lem}

\begin{lem}[Wick's Theorem] \label{lem:wick} Let $A$ be a real, symmetric positive-definite matrix as above, and let $B,C$ be symmetric matrices. Then the following relations hold:

$$
\int_{\Bbb R^n} (B\bold u,A\bold u)e^{-\frac{1}{2}(\bold u,A\bold u)}d^n \bold u={\rm Tr}(B)(\det A)^{-1/2};
$$

$$
\int_{\Bbb R^n} (B\bold u,A\bold u)(C\bold u,A\bold u)e^{-\frac{1}{2}(\bold u,A\bold u)}d^n \bold u=({\rm Tr}(B){\rm Tr}(C)+2{\rm Tr}(BC))(\det A)^{-1/2}. 
$$
\end{lem} 

\begin{proof}
These results are classical and well-known.
\end{proof}

Now, $x_i,y_i,z_i,t_i$, $1\le i\le n$ be real variables, and define a symmetric matrix $A$ by 
$$
A_{ij}=\frac{(x_i-x_j)(y_i-y_j)(z_i-z_j)}{t_i-t_j}, \ i \neq j
s$$
and $A_{ii} = 0$, as in Gaiotto's formula. The diagonal entries of $a_{ij}$ are defined to be zero. This is obviously a real, symmetric matrix, but it is not necessarily positive-definite. Nevertheless, we may still use the Gaussian integral representation and Lemma \ref{LemDiffIntAbsolute} to compute the action of the Gaudin operators $G_r^y$ on $(\det A)^{-1/2}$: 

\begin{lem} \label{LemGaudinOnA}
Let the matrix $A$ and the Gaudin operators $G_r^y$ be defined as above. The following formula holds: 
\begin{multline}
    \label{EqGaudinTraces}
4G_r^y (\det A)^{-1/2} = \sum_{s\ne r}\tfrac{y_r-y_s}{t_r-t_s}\sum_{m \neq r} z_m \bigg[-\sum_{p \neq s} \Big( (y_r-y_s)z_p({\rm Tr}(D_{rm}A^{-1}){\rm Tr}(D_{sp}A^{-1})+\\
\left.2{\rm Tr}(D_{rm}A^{-1}D_{sp}A^{-1})\right.) \Big) -2{\rm Tr}((D_{rm}-D_{sm})A^{-1})\bigg].
\end{multline}

\end{lem}

\begin{proof}
Let $E_{ij}$ be the elementary matrix with a single $1$ at the $i, j$ position, we may write 
\begin{align} 
\begin{split}
    D_{km} &= -\frac{x_k-x_m}{t_k-t_m}(E_{km}+E_{mk}), \ \ k \neq m; \\ \label{EqDefD}
    D_{kk} &= \sum_{m\ne k}\frac{x_k-x_m}{t_k-t_m}(E_{km}+E_{mk}), 
\end{split}
\end{align}
and in this case, 
\begin{equation}
    A=\sum_{k,m=1}^n y_kz_mD_{km} \label{eq:defA}
\end{equation}
Now, \textit{define} $A$ by \eqref{eq:defA}, leaving the $D_{km}$ to be arbitrary matrices. If $A$ is positive-definite, then Lemma \ref{lem:gaussian} tells us that
\[ 
(\det A)^{-1/2}=\int_{\Bbb R^n}e^{-\frac{1}{2}(\bold u,A\bold u)}d^n \bold u=
\int_{\Bbb R^n}e^{-\frac{1}{2}\sum_{k,m} y_kz_m(\bold u,D_{km}\bold u)}d^n \bold u \]

Recall that the Gaudin operators have the form 
$$
G_r=\sum_{s\ne r}\frac{1}{t_r-t_s} \left[ -(y_r-y_s)^2\partial_r\partial_s+(y_r-y_s)(\partial_r-\partial_s) \right]
$$
where we neglect an overall constant (which acts manifestly symmetrically on the variables $x, y, z$).

By Lemmas \ref{LemDiffIntAbsolute} and \ref{LemExponentIsGood}, we may safely consider the action of the Gaudin operator under the integral sign. Differentiating under the integral sign, we find
\begin{multline*}
4G_r^y (\det A)^{-1/2}=
\int_{\Bbb R^n} d^n \bold u \ e^{-\frac{1}{2}(\bold u,A\bold u)}\Bigg[\sum_{s\ne r}\tfrac{y_r-y_s}{t_r-t_s} \times  \\ \sum_m z_m \bigg(-\sum_{p}(y_r-y_s)z_p(D_{rm}\bold u,\bold u)(D_{sp}\bold u,\bold u)-
 2((D_{rm}-D_{sm})\bold u,\bold u) \bigg)\Bigg]
\end{multline*}
An application of Lemma \ref{lem:wick} therefore yields
\begin{multline*}
4G_r^y (\det A)^{-1/2} = \sum_{s\ne r}\tfrac{y_r-y_s}{t_r-t_s}\sum_m z_m (-\sum_{p}(y_r-y_s)z_p\left({\rm Tr}(D_{rm}A^{-1}){\rm Tr}(D_{sp}A^{-1})+\right.\\
\left.2{\rm Tr}(D_{rm}A^{-1}D_{sp}A^{-1})\right.)-2{\rm Tr}((D_{rm}-D_{sm})A^{-1})).
\end{multline*}

This is a purely algebraic expression in matrices $A$. Since it holds over all positive-definite matrices $A$, the equality holds for \textit{all} matrices $A$. In particular, we may take the $D_{km}$ as in Eqn. \ref{EqDefD}. This completes the proof. 
\end{proof}

Now, define matrix elements $b_{ij}$ as follows: 
\[ b_{ij}:=(A^{-1})_{ij}, \]
so that
\begin{equation}
\label{EqBInverseToA}
\sum_{k\ne i}\frac{(y_i-y_k)(z_i-z_k)(x_i-x_k)}{(t_i-t_k)}b_{kj}=
\delta_{ij}.
\end{equation} 
We may now massage \eqref{EqGaudinTraces} into a more convenient form.

\begin{lem} \label{LemOmega}
The following formula holds: 
    $$
(\det A)^{1/2}G_r^y (\det A)^{-1/2} := \Omega_r := \sum_{s\ne r}\frac{\Omega_{rs}}{t_r-t_s}, 
$$
where $\Omega_{rs}$ is defined by 
$$
\Omega_{rs}=\sum_{m\ne r,p\ne s}\tfrac{(y_r-y_s)^2(z_r-z_m)(z_p-z_s)(x_r-x_m)(x_s-x_p)}{(t_r-t_m)(t_s-t_p)}
(b_{mr}b_{sp}+b_{ms}b_{rp}+b_{mp}b_{sr})
$$
$$
-\sum_{m\ne r}\tfrac{(x_r-x_m)(y_r-y_s)(z_r-z_m)}{t_r-t_m}b_{rm}+\sum_{s\ne p}\tfrac{(x_s-x_p)(y_r-y_s)(z_s-z_p)}{t_s-t_p}b_{sp}.
$$
\end{lem}

\begin{proof}
We have 
\[D_{sp}A^{-1}=-\frac{x_s-x_p}{t_s-t_p}(E_{sp}+E_{ps})A^{-1},\] 
so the $(i,j)$-th element of this matrix is 
\[ (D_{sp}A^{-1})_{ij} = -\frac{x_s-x_p}{t_s-t_p}(\delta_{is}b_{pj}+\delta_{ip}b_{sj}).\]
Since $D_{ss}=-\sum_{p\ne s}D_{sp}$, we find 
\begin{align}
\begin{split}
    {\rm Tr}(D_{sp}A^{-1}) &= -2\frac{x_i-x_p}{t_i-t_p}b_{sp}, \ \ s\ne p; \\
    {\rm Tr}(D_{ss}A^{-1}) &= 2\sum_{p\ne s}\frac{x_i-x_p}{t_i-t_p}b_{sp}.
\end{split}
\end{align}
Note now that we may also write
\[ D_{sp}A^{-1} = -\frac{x_s-x_p}{t_s-t_p}\sum_j (E_{sj}b_{pj}+E_{pj}b_{sj}) . \]
When $r\ne m,s\ne p$ we have 
\begin{multline*}
    D_{rm}A^{-1}D_{sp}A^{-1}=\frac{(x_r-x_m)(x_s-x_p)}{(t_r-t_m)(t_s-t_p)}\sum_{j,k} (E_{sj}b_{pj}+E_{pj}b_{sj})(E_{rk}b_{mk}+E_{mk}b_{rk})=\\
    \frac{(x_r-x_m)(x_s-x_p)}{(t_r-t_m)(t_s-t_p)}\sum_{j,k}\delta_{jr}(E_{sk}b_{pj}b_{mk}+E_{pk}b_{sj}b_{mk})+\delta_{jm}(E_{sk}b_{pj}b_{rk}+E_{pk}b_{sj}b_{rk})=\\
    \frac{(x_r-x_m)(x_s-x_p)}{(t_r-t_m)(t_s-t_p)}\sum_k(E_{sk}(b_{pr}b_{mk}+b_{pm}b_{rk})+E_{pk}(b_{sr}b_{mk}+b_{sm}b_{rk}))
\end{multline*}
We conclude that
$$
{\rm Tr}(D_{rm}A^{-1}D_{sp}A^{-1})=2F_{rmsp}, \ \ r\neq m, \ \ s \neq p
$$
where 
$$
F_{rmsp}=\frac{(x_r-x_m)(x_s-x_p)}{(t_r-t_m)(t_s-t_p)}(b_{ms}b_{pr}+b_{mp}b_{sr}).
$$
In the case when $r=m$ we use $D_{rr}=-\sum_{m\ne r}D_{rm}$ to get the same formula with
$$
F_{rrsp}=-\sum_{m\ne r}\frac{(x_r-x_m)(x_s-x_p)}{(t_r-t_m)(t_s-t_p)}(b_{ms}b_{pr}+b_{mp}b_{sr}).
$$
Similarly, if $p=s$, we get
$$
F_{rmss}=-\sum_{p\ne s}\frac{(x_r-x_m)(x_s-x_p)}{(t_r-t_m)(t_s-t_p)}(b_{ms}b_{pr}+b_{mp}b_{sr}).
$$
Finally,
$$
F_{rrss}=\sum_{m\ne r,p\ne s}\frac{(x_r-x_m)(x_s-x_p)}{(t_r-t_m)(t_s-t_p)}(b_{ms}b_{pr}+b_{mp}b_{sr}).
$$
This allows us to write the following:
\begin{multline}
\label{EqTracesComputed}
{\rm Tr}(D_{rm}A^{-1}){\rm Tr}(D_{sp}A^{-1})+2{\rm Tr}(D_{rm}A^{-1}D_{sp}A^{-1})= \\ 4\frac{(x_r-x_m)(x_s-x_p)}{(t_r-t_m)(t_s-t_p)}(b_{rm}b_{sp}+b_{ms}b_{pr}+b_{mp}b_{sr})
\end{multline}
when $m\ne r$, $p\ne s$.

Let $m\ne r$, $p\ne s$. We have four contributions to the 
\[\frac{(x_r-x_m)(x_s-x_p)}{(t_r-t_m)(t_s-t_p)}(b_{rm}b_{sp}+b_{ms}b_{pr}+b_{mp}b_{sr})\] 
term when we combine~\eqref{EqGaudinTraces} and~\eqref{EqTracesComputed}: $-z_mz_p$ as above, $z_mz_s$ when $p=s$, $z_rz_p$ when $m=r$ and $-z_rz_s$ when $p=s$ and $r=m$. Overall, this gives $(y_r-y_s)^2(z_r-z_m)(z_p-z_s)$.

We similarly compute degree $1$ pieces:

\begin{multline*}-\tfrac12\sum_m z_m \Tr(D_{rm}-D_{sm})A^{-1}=\sum_{m\ne r}z_m\frac{x_r-x_m}{t_r-t_m}b_{rm}-\sum_{m\ne r} z_r\frac{x_r-x_m}{t_r-t_m}b_{rm}-\\
\sum_{m\ne s}z_m\frac{x_s-x_m}{t_s-t_m}b_{sm}+\sum_{m\ne s}z_s\frac{x_s-x_m}{t_s-t_m}b_{sm}=\\
-\sum_{m \ne r}\frac{(z_r-z_m)(x_r-x_m)}{t_r-t_m}b_{rm}+\sum_{m\ne s}\frac{(z_s-z_m)(x_s-x_m)}{t_s-t_m}b_{sm}
\end{multline*}

Combining the above equations with~\eqref{EqGaudinTraces} we get that

$$
(\det A)^{1/2}G_r^y (\det A)^{-1/2} := \Omega_r := \sum_{s\ne r}\frac{\Omega_{rs}}{t_r-t_s}, 
$$
where $\Omega_{rs}$ is defined by 
$$
\Omega_{rs}=\sum_{m\ne r,p\ne s}\tfrac{(y_r-y_s)^2(z_r-z_m)(z_p-z_s)(x_r-x_m)(x_s-x_p)}{(t_r-t_m)(t_s-t_p)}
(b_{mr}b_{sp}+b_{ms}b_{rp}+b_{mp}b_{sr})
$$
$$
-\sum_{m\ne r}\tfrac{(x_r-x_m)(y_r-y_s)(z_r-z_m)}{t_r-t_m}b_{rm}+\sum_{s\ne p}\tfrac{(x_s-x_p)(y_r-y_s)(z_s-z_p)}{t_s-t_p}b_{sp}.
$$
This completes the proof.
\end{proof}

Now, decompose $\Omega_{r} = \Omega_{r}^{(1)} + \Omega_{r}^{(2)}$, where $\Omega_{r}^{(i)}$ contains terms of degree $i$ in the symbols $b_{mk}$. In particular, 
\[ \Omega_r^{(2)} = \sum_{m\ne r,p\ne s}\tfrac{(y_r-y_s)^2(z_r-z_m)(z_p-z_s)(x_r-x_m)(x_s-x_p)}{(t_r-t_m)(t_s-t_p)}
(b_{mr}b_{sp}+b_{ms}b_{rp}+b_{mp}b_{sr});  \]
\[ \Omega_r^{(1)} = -\sum_{m\ne r}\tfrac{(x_r-x_m)(y_r-y_s)(z_r-z_m)}{t_r-t_m}b_{rm}+\sum_{s\ne p}\tfrac{(x_s-x_p)(y_r-y_s)(z_s-z_p)}{t_s-t_p}b_{sp}. \]
Our strategy will be to rearrange $\Omega_{r}^{(2)}$ until we obtain an expression which can be shown to be symmetric in $y$ and $z$. Along the way, we will obtain additional degree-1 terms, which we will denote by $\Sigma_{r}^{(1)}$. Finally, we will demonstrate that $\Omega_{r}^{(1)} + \Sigma_{r}^{(1)}$ is symmetric in $y, z$. 

We begin by noticing that
\[(y_r-y_s)^2=(y_r-y_s)(y_r-y_p)+(y_r-y_s)(y_p-y_s).\] 
Moreover, we find that
\[ \sum_{p \neq s} \frac{(y_p-y_s)(z_p-z_s)(x_s-x_p)}{t_s-t_p}(b_{mr}b_{sp}+b_{ms}b_{rp}+b_{mp}b_{sr}) = b_{rm} + \delta_{sr} b_{ms} + \delta_{ms} b_{rs} . \] 
We then see that 
\[ \sum_{m\ne r,p\ne s}\frac{(y_r-y_s)(y_p - y_s)(z_r-z_m)(z_p-z_s)(x_r-x_m)(x_s-x_p)}{(t_r-t_m)(t_s-t_p)}
(b_{mr}b_{sp}+b_{ms}b_{rp}+b_{mp}b_{sr}) =  \]
\[ \sum_{m, s \neq r} \frac{(x_r - x_m)(y_r - y_s)(z_r - z_m)}{(t_r - t_m)(t_r - t_s)} b_{rm} + \cdots := \Sigma_{r}^{(1), 1},  \]
where the terms $\cdots$ are manifestly symmetric in $y, z$. 

%We are left with \begin{multline*}
%\sum_{m,p} (y_r-y_s)(y_r-y_p)(z_r-z_m)(z_p-z_s)(x_r-x_m)(x_s-x_p)\frac{1}{{(t_r-t_m)(t_s-t_p)}}\\
% (b_{mr}b_{sp}+b_{ms}b_{rp}+b_{mp}b_{sr}).
%\end{multline*} Since we are taking sum $\frac{\Omega_{rs}}{t_r-t_s}$ we should multiply this by $\frac{1}{t_r-t_s}$ and get

The remaining degree-two piece looks as follows: 
\begin{multline*} {\Omega'}_r^{(2)} = \sum_{s \neq r} \sum_{m \neq r, p \neq r, s} (y_r-y_s)(y_r-y_p)(z_r-z_m)(z_p-z_s)(x_r-x_m)(x_s-x_p) \\
\times \frac{1}{(t_r-t_m)(t_s-t_p)(t_r-t_s)} (b_{mr}b_{sp}+b_{ms}b_{rp}+b_{mp}b_{sr}), 
\end{multline*}
so that 
\[ \Omega_r^{(2)} = {\Omega'}_r^{(2)} + \Sigma_{r}^{(1), 1}. \] 
Note that the $y_r-y_p$ term allows us to add a condition $p\neq r$ without changing the sum. Hence the summation is symmetric with respect to the exchange of summation indices $s\leftrightarrow p$. Thus, we may symmetrize with respect to change $s\leftrightarrow p$. Note that the numerator is symmetric with respect to $s\leftrightarrow p$. The denominator transforms as follows:
\[\frac{1}{(t_r-t_m)(t_s-t_p)(t_r-t_s)}  \implies \]
\[  \frac{1}{2} \left[ \frac{1}{(t_r-t_m)(t_s-t_p)(t_r-t_s)}+\frac{1}{(t_r-t_m)(t_p-t_s)(t_r-t_p)} \right]\]
\[ = \frac{1}{2} \frac{1}{(t_r-t_m)(t_s-t_p)}\left[ \frac{1}{t_r-t_s}-\frac{1}{t_r-t_p} \right]= \frac{1}{2}\frac{1}{(t_r-t_m)(t_r-t_s)(t_r-t_p)} . \]

Now, we write 
\[y_r-y_p=(y_r-y_m)+(y_m-y_p).\]
and we notice that
\[ \sum_{m \neq r} \frac{(y_r-y_m)(z_r-z_m)(z_r-x_m)}{t_r-t_m}(b_{mr}b_{sp}+b_{ms}b_{rp}+b_{mp}b_{sr}) = b_{sp} + \delta_{rs} b_{rp} + \delta_{rp} b_{rs} . \]
We see then that 
\begin{multline*} \frac{1}{2} \sum_{s \neq r} \sum_{m \neq r, p \neq r, s} (y_r-y_s)(y_r-y_m)(z_r-z_m)(z_p-z_s)(x_r-x_m)(x_s-x_p) \\
\times \frac{1}{(t_r-t_m)(t_r-t_p)(t_r-t_s)} (b_{mr}b_{sp}+b_{ms}b_{rp}+b_{mp}b_{sr})
\end{multline*}
\[ = - \frac{1}{2} \sum_{p, s \neq r} \frac{(x_s - x_p)(y_r - y_s)(z_s - z_p)}{(t_r - t_p)(t_r - t_s)} b_{sp} + \cdots :=  \Sigma_{r}^{(1), 2}  , \]
where the $\cdots$ are terms manifestly symmetric in $y, z$. 

The degree-two term that we are left with is then 
\begin{multline*} {\Omega''}_{r}^{(2)} =  \frac{1}{2} \sum_{m,s,p \neq r}(y_r-y_s)(y_m-y_p)(z_r-z_m)(z_p-z_s)(x_r-x_m)(x_s-x_p)\\
\frac{1}{(t_r-t_m)(t_r-t_s)(t_r-t_p)} (b_{mr}b_{sp}+b_{ms}b_{rp}+b_{mp}b_{sr}) 
\end{multline*}
\[ :=  \frac{1}{2} \sum_{m,s,p \neq r}A_{rmsp}(b_{mr}b_{sp}+b_{ms}b_{rp}+b_{mp}b_{sr}).\]
so that
\[ \Omega_r^{(2)} = {\Omega'}_r^{(2)} + \Sigma_r^{(1), 1} = {\Omega''}_r^{(2)} + \Sigma_r^{(1), 1} + \Sigma_r^{(1), 2}. \]
We note that the presence of terms $z_p - z_s$ in the numerator allows us to safely drop the condition $p \neq s$ in the sums. 

At this point, we are in a position to explicitly show the symmetry of the degree-two piece in $y, z$. 

\begin{lem} \label{LemDeg2}
    The quantity ${\Omega''}_r^{(2)}$ is symmetric in $y, z$. 
\end{lem}

\begin{proof}
We begin by noting that the terms $b_{mr}b_{sp}+b_{ms}b_{rp}+b_{mp}b_{sr}$ are invariant under any cyclic permutation of the indices $m, s, p$. By a sum over cyclic permutations of $m, s, p$, we see that for fixed $m,s,p$ the coefficient on $b_{mr}b_{sp}$ is $\frac{1}{2}(A_{rmsp}+A_{rspm}+A_{rpms})$. We will prove that $A_{rmsp}+A_{rspm}+A_{rpms}$ is symmetric in $y$ and $z$.

Define $H_{rmsp}(x)$ as follows: 
\[ H_{rmsp}(x) =(x_r-x_m)(x_s-x_p) . \] 
We have 
\[A_{rmsp}=\frac{H_{rspm}(y)H_{rmsp}(x)H_{rmsp}(z)}{(t_r-t_m)(t_r-t_s)(t_r-t_p)}.\] 
Since the denominator is symmetric in $m,p,s$ we should prove that the sum of cyclic shifts of $H_{rspm}(y)H_{rmsp}(x)H_{rmsp}(z)$ in these three indices is symmetric with respect to $y,z$. Explicitly, we wish to show that
\[ H_{rspm}(y)H_{rmsp}(x)H_{rmsp}(z) + H_{rmsp}(y)H_{rpms}(x)H_{rpms}(z) + H_{rpms}(y)H_{rspm}(x)H_{rspm}(z)  \]
is symmetric in $y, z$. 

We first note by direct computation that 
\[ H_{rmsp}(x)+H_{rspm}(x)+H_{rpms}(x)=0,  \]
so we may replace $H_{rpms}(x)$ with $-H_{rmsp}(x)-H_{rspm}(x)$. We find that
\[ H_{rspm}(y)H_{rmsp}(x)H_{rmsp}(z) + H_{rmsp}(y)H_{rpms}(x)H_{rpms}(z) + H_{rpms}(y)H_{rspm}(x)H_{rspm}(z) \]
\[ = C^{(1)}_{rmsp} + C^{(2)}_{rmsp}, \]
where
\[ C^{(1)}_{rmsp} = H_{rmsp}(x)(H_{rspm}(y)H_{rmsp}(z)-H_{rmsp}(y)H_{rpms}(z));  \]
\[ C^{(2)}_{rmsp} = H_{rspm}(x)(H_{rpms}(y)H_{rspm}(z)-H_{rmsp}(y)H_{rpms}(z)) .\]

We will prove that the first cyclic shift $C_{rmsp}^{(1)}$ is symmetric in $y,z$; the proof that $C_{rmsp}^{(2)}$ is symmetric in $y, z$ is entirely analogous. We wish to show that
\[H_{rspm}(y)H_{rmsp}(z)-H_{rmsp}(y)H_{rpms}(z)=H_{rspm}(z)H_{rmsp}(y)-H_{rmsp}(z)H_{rpms}(y).\]
Indeed, the leftmost and the rightmost term have sum 
\[H_{rmsp}(z)(H_{rspm}(y)+H_{rpms}(y))=-H_{rmsp}(z)H_{rmsp}(y).\] 
The middle two terms have sum 
\[H_{rmsp}(y)(H_{rpms}(z)+H_{rspm}(z))=-H_{rmsp}(y)H_{rmsp}(z),\] 
hence the left-hand side is equal to the right-hand side. Thus, the first cyclic shift is symmetric. The second cyclic shift is done similarly. We are therefore done.
\end{proof}

Now, we focus on the degree-one pieces.

\begin{lem} \label{LemDeg1}
In particular, define by $\Sigma_{r}^{(1)}$ the remaining terms of degree 1 after we have transformed the degree-two piece as above:
\[ \Sigma^{(1)}_{r} = \Sigma_{r}^{(1),1} + \Sigma_{r}^{(1),2} \]
Then $\Omega_r^{(1)} + \Sigma_r^{(1)}$ is symmetric in $y, z$. 
\end{lem}

\begin{proof}
In particular, let us denote by $\Sigma_{r}^{(1)}$ the remaining terms of degree 1 after we have transformed the degree-two piece as above. We have computed the terms of $\Sigma_{r}^{(1)}$ already: 
\[ \Sigma^{(1)}_{r} = \Sigma_{r}^{(1),1} + \Sigma_{r}^{(1),2} \]
\[ = \sum_{m, s \neq r} \frac{(x_r - x_m)(y_r - y_s)(z_r - z_m)}{(t_r - t_m)(t_r - t_s)} b_{rm} - \frac{1}{2} \sum_{p, s \neq r} \frac{(x_s - x_p)(y_r - y_s)(z_s - z_p)}{(t_r - t_p)(t_r - t_s)} b_{sp} . \]
On the other hand, the degree-one terms present in the original $\Omega_{r}^{(1)}$ are as follows: 
\[ \Omega_{r}^{(1)} = -\sum_{m, s \neq r}  \frac{(x_r - x_m)(y_r - y_s)(z_r - z_m)}{(t_r - t_m)(t_r - t_s)} b_{rm} + \sum_{p, s \neq r} \frac{(x_s - x_p)(y_r - y_s)(z_s - z_p)}{(t_r - t_s)(t_s - t_p)} b_{sp} .\]
The combined degree-one term we require is $\Omega_{rs}^{(1)} + \Sigma_{rs}^{(1)}$. The first terms of $\Omega_{rs}^{(1)}$ and $\Sigma_{rs}^{(1)}$ cancel, so we are left with
\[ \Omega_{r}^{(1)} + \Sigma_{r}^{(1)} = \sum_{p, s \neq r} \frac{(x_s - x_p)(y_r - y_s)(z_s - z_p)}{(t_r - t_s)} \left[ \frac{1}{t_s - t_p} - \frac{1}{2} \frac{1}{t_r - t_p} \right] b_{sp}  = \]
\[ \frac{1}{2} \sum_{p, s \neq r} \frac{(x_s - x_p)(y_r - y_s)(z_s - z_p)(2t_r - t_s - t_p)}{(t_r - t_p)(t_r - t_s)(t_s - t_p)} b_{sp} . \]
Since $b_{sp}$ is symmetric under $s \leftrightarrow p$ for the summed indices $s, p$, we may symmetrize its coefficient. We find that
\[ \frac{(x_s - x_p)(y_r - y_s)(z_s - z_p)(2t_r - t_s - t_p)}{(t_r - t_p)(t_r - t_s)(t_s - t_p)} \ \longmapsto  \] 
\[ -\frac{1}{2} \frac{(x_s - x_p)(y_s - y_p)(z_s - z_p)(2t_r - t_s - t_p)}{(t_r - t_p)(t_r - t_s)(t_s - t_p)}. \]
But we now see that this summand is symmetric in $y, z$. Thus, the entirety of $\Omega_{r}$ is indeed symmetric in $y, z$, so $G_r^y$ and $G_r^z$ act identically. We are done. 
\end{proof}

Now, we can put all of the pieces together. Recall that
\[ \Omega_r = \Omega_r^{(2)} + \Omega_r^{(1)} = {\Omega''}_r^{(2)} + (\Omega_r^{(1)} + \Sigma_r^{(1)}). \]
By Lemmas \ref{LemDeg2} and \ref{LemDeg1}, ${\Omega''}_r^{(2)}$ and $\Omega_r^{(1)} + \Sigma_r^{(1)}$ are each symmetric in $y, z$, so $\Omega_r$ must also be. By Lemma \ref{LemOmega}, $G_r^y(\det A)^{-1/2}$ is therefore symmetric in $y, z$, so 
\[ G_r^y(\det A)^{-1/2} = G_r^z(\det A)^{-1/2}, \]
exactly as we wanted to show. We are done. 

\subsection{$\lambda$-Twisted Generalization}
\label{subsec:twist}

A natural generalization of our formula incorporates a so-called $\lambda$-\textit{twist}. In \cite{EFK1}, $\lambda$-twisted versions of the Gaudin operators are defined: 
\[ G_i^\lambda = \sum_{j \neq i} \frac{1}{t_i - t_j} \left[ - (y_i - y_j)^2 \p_i \p_j + (y_i - y_j)(\lambda_i \p_i - \lambda_j \p_j) + \frac{\lambda_i\lambda_j}{2} \right]. \]

In particular, we modify the Gaussian integral representation of $K_3$ as follows to obtain a kernel $K_3^\lambda$:
\[ K_3^\lambda(x, y, z; t) = \int_{\B{C}}  d^m u d^m \bar{u} \left[ \prod_{i = 0}^m  \abs{u_i} ^{-2(\l_i + 1)} \right] e^{(\mathbf{u}, A\mathbf{u}) - (\mathbf{\bar{u}}, \bar{A}\mathbf{\bar{u}})},  \]
This integral can be sensibly defined in the sense of tempered distributions as per the techniques in Appendices \ref{sec:diffint} and \ref{sec:tempereddist}. In particular, we show in \ref{prop:diffoponfourier} that we may take the action of the holomorphic Gaudin operator inside the integral sign. Now, let $n_i = -\l_i - 1$; the action of the (nonconstant part of the) holomorphic Gaudin operator $G_i^\l$ on the holomorphic part of the above integral yields the following:
\[ G_r^\l \left( \int d^n u \left[ \prod_{i = 1}^n u_i^{n_i} \right] e^{(\mathbf{u}, A \mathbf{u})}  \right) = \]
\[   \int d^n u \left[ \prod_{i = 1}^n u_i^{n_i} \right] \sum_{s \neq r} \frac{1}{t_r - t_s} \left[ - (y_r - y_s)^2 \p_r \p_s + (y_r - y_s)((n_s + 1) \p_r - (n_r + 1) \p_s)  \right] e^{(\mathbf{u}, A \mathbf{u})}  \]
\[ \int d^n u \left[ \prod_{i = 1}^n u_i^{n_i} \right] e^{(\mathbf{u}, A \mathbf{u})} \times \]
\[ = \Bigg[ -\sum_{s\neq r}\sum_{m\neq r, p \neq s} \frac{(y_r - y_s)^2(x_r - x_m)(x_s - x_p)(z_r - z_m)(z_s - z_p)}{(t_r - t_s)(t_r - t_m)(t_s - t_p)} u_r u_s u_m u_p \]
\[ + \sum_{s \neq r} (n_s + 1) \sum_{m \neq r} \frac{(x_r - x_m)(y_r - y_s)(z_r - z_m)}{(t_r - t_s)(t_r - t_m)} u_r u_m  \]
\begin{equation} \label{eqn:twistedgaudin} - (n_r + 1) \sum_{s \neq r} \sum_{p \neq s} \frac{(x_s - x_p)(y_r - y_s)(z_s - z_p)}{(t_r - t_s)(t_s - t_p)} u_s u_p \Bigg] .
\end{equation}
At this point, we wish to integrate the second and third sums by parts, which is legitimate by Proposition \ref{prop:integrationbyparts}. In particular, we recognize that
\[ (n_s + 1) u_s^{n_s} = \frac{d}{du_s}  u_s^{n_s + 1}. \]
In integrating by parts, we integrate this term and differentiate the coefficient it multiplies (including the exponential) as follows: 
\[ \int u \, dv \implies - \int v \, du. \]
Taking
\[  u = e^{(\mathbf{u}, A\mathbf{u})} \sum_{m \neq r} \frac{(x_r - x_m)(y_r - y_s)(z_r - z_m)}{(t_r - t_s)(t_r - t_m)} u_r u_m , \quad dv = (n_s + 1) u_s^{n_s} \cdot \prod_{i \neq s} u_i^{n_i} du_s,  \]
we find that 
\[  u \, dv = -  \prod_{i = 1}^n u_i^{n_i} du_s \cdot (n_s + 1) \sum_{m \neq r} \frac{(x_r - x_m)(y_r - y_s)(z_r - z_m)}{(t_r - t_s)(t_r - t_m)} u_r u_m  ; \]
\[ - v \, du = \prod_{i = 1}^n u_i^{n_i} du_s \cdot \sum_{m \neq r, p \neq s} \frac{(x_r - x_m)(x_s - x_p)(y_r - y_s)(y_s - y_p)(z_r - z_m)(z_s - z_p)}{(t_r - t_s)(t_r - t_m)(t_s - t_p)} u_r u_s u_m u_p . \] 
Thus, dropping boundary terms as per Proposition \ref{prop:integrationbyparts}, we may effectively replace the second sum in brackets in \eqref{eqn:twistedgaudin} with 
\[ \sum_{s \neq r} (n_s + 1) \sum_{m \neq r} \frac{(x_r - x_m)(y_r - y_s)(z_r - z_m)}{(t_r - t_s)(t_r - t_m)} u_r u_m  \implies \]
\begin{equation} \label{eqn:secondsum}  -\sum_{s \neq r} \sum_{m \neq r, p \neq s} \frac{(x_r - x_m)(x_s - x_p)(y_r - y_s)(y_s - y_p)(z_r - z_m)(z_s - z_p)}{(t_r - t_s)(t_r - t_m)(t_s - t_p)} u_r u_s u_m u_p . \end{equation}
By analogous reasoning, 
\[ -(n_r + 1) \sum_{s \neq r} \sum_{p \neq s} \frac{(x_s - x_p)(y_r - y_s)(z_s - z_p)}{(t_r - t_s)(t_s - t_p)} u_s u_p  \implies \]
\begin{equation} \label{eqn:thirdsum}  \sum_{s \neq r} \sum_{m \neq r, p \neq s} \frac{(x_r - x_m)(x_s - x_p)(y_r - y_s)(y_r - y_m)(z_r - z_m)(z_s - z_p)}{(t_r - t_s)(t_r - t_m)(t_s - t_p)} u_r u_s u_m u_p  \end{equation}
The term in \eqref{eqn:secondsum} in now combines nicely with the first sum in brackets in \eqref{eqn:twistedgaudin}: 
\[ -\sum_{s\neq r}\sum_{m\neq r, p \neq s} \frac{(y_r - y_s)^2(x_r - x_m)(x_s - x_p)(z_r - z_m)(z_s - z_p)}{(t_r - t_s)(t_r - t_m)(t_s - t_p)} u_r u_s u_m u_p \]
\[ + \sum_{s \neq r} (n_s + 1) \sum_{m \neq r} \frac{(x_r - x_m)(y_r - y_s)(z_r - z_m)}{(t_r - t_s)(t_r - t_m)} u_r u_m = \]
\[ -\sum_{s \neq r} \sum_{m \neq r, p \neq s} \frac{(x_r - x_m)(x_s - x_p)(y_r - y_s)(y_r - y_p)(z_r - z_m)(z_s - z_p)}{(t_r - t_s)(t_r - t_m)(t_s - t_p)} u_r u_s u_m u_p . \]
Note that the term with $r = p$ vanishes due to the factor of $y_r - y_p$ in the numerator. Thus, the summation indices $s, p$ run over the same range. We may thus symmetrize the coefficient of $u_r u_s u_m u_p$ in the indices $s, p$. Note that everything within the sum is symmetric in $s, p$ except for $\frac{1}{(t_r - t_s)(t_s - t_p)}$. We find
\[ \frac{1}{(t_r - t_s)(t_s - t_p)} + \frac{1}{(t_r - t_p)(t_p - t_s)} = \frac{1}{(t_r - t_s)(t_r - t_p)}, \]
and we obtain
\[ - \frac{1}{2} \sum_{s \neq r} \sum_{m \neq r, p \neq r} \frac{(x_r - x_m)(x_s - x_p)(y_r - y_s)(y_r - y_p)(z_r - z_m)(z_s - z_p)}{(t_r - t_s)(t_r - t_m)(t_r - t_p)} u_r u_s u_m u_p. \]
Now, write
\[ y_r - y_p = y_r - y_m + y_m - y_p . \]
Focus on the piece corresponding to $y_m - y_p$: 
\[  - \frac{1}{2} \sum_{m, s, p \neq r} \frac{(x_r - x_m)(x_s - x_p)(y_r - y_s)(y_m - y_p)(z_r - z_m)(z_s - z_p)}{(t_r - t_s)(t_r - t_m)(t_s - t_p)} u_r u_m u_s u_p . \]
Using the notation of Sec. \ref{subsec:mainproof}, the coefficient of $u_r u_m u_s u_p$ is given by
\[ A_{rmsp}=\frac{H_{rspm}(y)H_{rmsp}(x)H_{rmsp}(z)}{(t_r-t_m)(t_r-t_s)(t_r-t_p)} . \]
We already showed that $H_{rspm}(y)H_{rmsp}(x)H_{rmsp}(z)$ is symmetric in $y, z$ when summed over cyclic permutations of the indices $s, m, p$. The relevant computations are performed in Sec. \ref{subsec:mainproof}. The same computation follows through in this case, as the term $u_r u_m u_s u_p$ is invariant under these cyclic permutations. 

This leaves the following term: 
\[ - \frac{1}{2} \sum_{s \neq r} \sum_{m \neq r, p \neq r} \frac{(x_r - x_m)(x_s - x_p)(y_r - y_s)(y_r - y_m)(z_r - z_m)(z_s - z_p)}{(t_r - t_s)(t_r - t_m)(t_r - t_p)} u_r u_s u_m u_p. \]
Adding this to the term in \eqref{eqn:thirdsum} gives
\[ \sum_{s \neq r} \sum_{m \neq r, p \neq r} \tfrac{(x_r - x_m)(x_s - x_p)(y_r - y_s)(y_r - y_m)(z_r - z_m)(z_s - z_p)}{(t_r - t_s)(t_r - t_m)} \left[ \tfrac{1}{t_s - t_p} - \tfrac{1}{2} \tfrac{1}{t_r - t_p} \right] u_r u_s u_m u_p =\]
\[  \frac{1}{2} \sum_{s \neq r} \sum_{m \neq r, p \neq r} \tfrac{(x_r - x_m)(x_s - x_p)(y_r - y_s)(y_r - y_m)(z_r - z_m)(z_s - z_p)(2t_r - t_s - t_p)}{(t_r - t_s)(t_r - t_p)(t_s - t_p)(t_r - t_m)} u_r u_s u_m u_p  . \]
Symmetrizing the coefficient of $u_r u_s u_m u_p$ in $s$ and $p$, we obtain
\[ -\frac{1}{4} \sum_{s \neq r} \sum_{m \neq r, p \neq r} \tfrac{(x_r - x_m)(x_s - x_p)(y_s - y_p)(y_r - y_m)(z_r - z_m)(z_s - z_p)(2t_r - t_s - t_p)}{(t_r - t_s)(t_r - t_p)(t_s - t_p)(t_r - t_m)} u_r u_s u_m u_p  , \]
which is symmetric in $y, z$. We are done. 

\section{Discussion}
\label{sec:discussion}

In summary, we have proved a conjectured formula for an intertwining kernel for the quantum Hitchin Hamiltonians in genus zero (Gaudin operators) over $\mathbb{C}$ with $\abs{S} \geq 4$ marked points. We have additionally demonstrated a natural generalization of our formula to account for a possible ``$\lambda$-twisting.'' With all of this in mind, we now assess the various implications of our result. 

First of all, it is of natural interest to consider the analytic geometric Langlands program over fields other than $\mathbb{C}$. The case of $\mathbb{R}$ is handled very similarly to $\mathbb{C}$. Gaiotto \cite{Gaiotto2021} produces a very similar formula for an intertwining kernel for the \textit{real} Gaudin system in three sets of variables:
\[ K_3^{\mathbb{R}}(x, y, z; t) = \frac{1}{\sqrt{\abs{\det A}}} . \]
The argument from the previous section goes through essentially unchanged to demonstrate this formula as well. We represent $(\det A)^{-1/2}$ by a Gaussian integral, act on the integral by the Gaudin operator, and apply Wick's theorem. The resulting expression can be massaged into a form which is manifestly symmetric in the $x, y, z$ variables. 

Over local fields, the story is more complicated. We no longer have at our disposal the differential Gaudin operators, and to understand the spectrum of the Hecke operators, it would thus be an important stop to construct a Gaiotto kernel over local fields $F$. One might naturally guess something of the following form:  
\[ K_3^{F}(x, y, z; t) = \frac{\theta(\det A)}{\sqrt{\Abs{\det A}}}, \]
where $\Abs{\cdot}$ denotes a (non-Archimedean) absolute value over $F$. Here we define $\theta(x) = 1$ if $x$ is a square in $F$ and $\theta(x) = 0$ otherwise.

To show that $K_3^F$ intertwines the Hecke operators over \textit{all} fields $F$ is therefore an important next step. We will show this with a calculation of the action of the Hecke operators on $K_3^F$ in the upcoming work~\cite{klyuev22}.

%One can argue for this roughly as follows. Since the joint spectrum of the Gaudin and Hecke operators over $\mathbb{C}$ is simple \cite{EFK3}, we conclude that the Gaiotto kernel intertwines the action of the Hecke operators over $\mathbb{C}$. The explicit formula for the Hecke operators is known \cite{EFK3}, and from this, it may be argued that the action of these operators on $K_3^F$ can be represented in terms of elliptic integral invariants $I, J$ \cite{Gaiotto2021}, which are symmetric under the interchange of the three variables. 

%Since the above result is a purely formal one over $\mathbb{C}$, it should in principle hold over all other fields $F$. 

\section*{Acknowledgements}

SR would like to thank Pavel Etingof for suggesting this project and providing much guidance along the way. SR would also like to thank Davide Gaiotto for insightful discussions. DK would like to thank Pavel Etingof for suggesting this project, helping us along the way and suggesting a way to make sense of twisted Gaiotto kernel in terms of tempered distributions. Both authors want to thank David Jerison, Ankur Moitra, and the SPUR REU program for funding this research and for organizing meetings. 

\iffalse
\section*{Acknowledgements}

I would like to thank Pavel Etingof for suggesting this project and providing much guidance and insight along the way. I would also like to thank David Jerison and Ankur Moitra for organizing the SPUR program and several insightful discussions. Finally, and most importantly, I would like to thank my mentor, Daniil Klyuev, who helped me greatly and provided guidance throughout the course of this project.
\fi

\appendix

\section{Differentiating under the Integral Sign}
\label{sec:diffint}

In this section, we want to prove that 
\[ \frac{\partial}{\partial x_i}\int f(x,y)dy=\int\frac{\partial}{\partial x_i}f(x,y)dy \]
for sufficiently ``good'' functions $f$. This will allow us to make the analytic manipulations we use in Section~\ref{subsec:mainproof}. We will also use the result of this section in the next Section~\ref{sec:tempereddist}.

Since we are differentiating in one variable at a time, we will assume that $x$ is just one variable. Let $n$ be a positive integer, and suppose $f(x,y)$ is a smooth function defined on $U\times \RN^n$, where $U$ is an open subset of $\RN$. Let 
\[ g(x)=\int_{\RN^n} f(x,y) dy,\] where we assume that the right-hand side is absolutely convergent for all $x\in U$.
Let $x^*$ be a point in $U$. Also denote $\partial_x=\frac{\partial}{\partial x}$. %We fix the index $i$ from $1$ to $m$. In what follows, denote $\partial_i=\frac{\partial}{\partial x_i}$.

\begin{lem}
\label{LemDiffIntAbsolute}
\begin{enumerate}
    \item 
    Suppose that there exists a neighborhood $V$ of $x^*$ and a function $M(y)$ with the following property: for all $x\in V$ we have $\abs{\partial_x f(x,y)}+\abs{\partial_x^2 f(x,y)}\leq M(y)$ and $\int_{\RN^n} M(y)\leq\infty$. Then $g(x)$ has a derivative at point $x^*$ that can be computed as 
\[   g'(x^*) = \int_{\RN^n} \partial_{x} f(x^*,y) dy. \]
    \item
    Suppose  that $x\in \RN^m$, $U$ is an open subset of $\RN^m$, $f(x,y)$ is a smooth function defined on $U\times \RN^n$, $g(x)$ is defined as above. Let $x^*$ be a point in $U$. Assume that for any partial differential operator $D$ in $x$ with constant coefficients there exists a neighborhood $V_D$ of $x^*$ and a function $M_D(y)$ with the following property: for all $x\in V_D$ we have $\abs{D f(x,y)}\leq M_D(y)$ and $\int_{\RN^n} M_D(y)\leq\infty$. Then $g(x)$ is smooth at $x^*$ and $Dg(x)$ can be computed as $\int_{\RN^n} Df(x^*,y)dy$.
    
\end{enumerate}

\end{lem}
\begin{proof}
The second statement immediately follows from the first using induction on the order of $D$.

 By definition,
\[g'(x^*)=\lim_{\eps\to 0}\frac{1}{\eps}(g(x^*+\eps)-g(x^*)).\]
We have 
\[g(x^*+\eps )-g(x^*)=\int_{\RN^n} \left[ f(x^*+ x_i,y)-f(x^*,y) \right] dy=\eps\int_{\RN^n}\partial_x f(x^*+a(y),y)dy. \]
Here, we have used mean value theorem, and we have let $a$ be a map from $\RN^n$ to $[0,\eps]$.  Thus, 
\[ \frac{1}{\eps}(g(x^*+\eps x_i)-g(x^*)) = \int_{\RN^n}\partial_x f(x^*+x_i a(y),y)dy .\]
Using the mean value theorem again, we find that
\[ \int_{\RN^n}\partial_x f(x^*+ a(y),y)dy - \int_{\RN^n}\partial_x f(x^*,y)dy =   \eps\int_{\RN^n} \partial^2_x f(x^*+b(y),y)dy,\]
where $b$ is another map from $\RN^n$ to $[0,\eps]$. Using the condition in the statement of the lemma we see that the expression on the right hand side is at most $C\eps$, where $C$ is a constant positive number. Hence it tends to zero when $\eps$ tends to zero, and indeed 
\[  g'(x^*) = \int_{\RN^n}\partial_x f(x^*,y)dy. \]
\end{proof}

\begin{lem}
\label{LemExponentIsGood}
Let $D_{km}$ be fixed matrices, $A=\sum y_kz_m D_{km}$. Let $U$ be the set of all $x$ such that the corresponding matrix $A$ is positive definite, $x^*$ be an element of $U$.

Then function $f(x,\mathbf{u})=e^{-\frac{1}{2}(A\mathbf{u},\mathbf{u})}$ satisfies conditions of Lemma~\ref{LemDiffIntAbsolute} with $y=\mathbf{u}$. In particular, we can compute a partial differential operator $D(\sqrt{\det A})$ as $\int_{\RN^n}D(e^{-\frac{1}{2}(A\mathbf{u},\mathbf{u})})d\mathbf{u}$ times a constant.
\end{lem}
\begin{proof}
For any partial differential operator $D$ with constant coefficients we have $Df(x,\mathbf{u})=p(x,\mathbf{u})e^{-\frac{1}{2}(A\mathbf{u},\mathbf{u})}$, where $p$ is a polynomial in $x,\mathbf{u}$. In a neighborhood $V$ of $x^*$ we can say that $p(x,\mathbf{u})\leq C_1+C_2 \abs{\mathbf{u}}^N$, where $C_1,C_2,N$ are constant positive numbers. Since any two norms on $\RN^m$ are equivalent, we have that 
\[ (A\mathbf{u},\mathbf{u} )\geq \eps (\mathbf{u},\mathbf{u}) .\]
for any $\mathbf{u}$. In fact, we can choose $\eps > 0$ such that the above equation holds true for all $A$ in a neighborhood $V \ni x^*$. 

Consider $M(\mathbf{u})=(C_1+C_2\abs{\mathbf{u}}^N) e^{-\eps (\mathbf{u},\mathbf{u})}$. We just showed that $\abs{Df(x,\bu)}\leq M(\bu) $ for $x\in V$.
We also see that $\int_{\RN^n}M(\bu)d\bu<\infty$. It follows that $f(x,\bu)$ satisfies the conditions of Lemma~\ref{LemDiffIntAbsolute}.
\end{proof}

With the preceding two lemmas, we have established the validity of differentiating sufficiently convergent real integrals. This finishes the details of the proof in Section~\ref{subsec:mainproof}. In the next section, we will explain how to deal with the integrals arising in the proof of the twisted Gaiotto formula in Section~\ref{subsec:twist}. 

\section{Integrals in the Sense of Tempered Distributions} \label{sec:tempereddist}

In this section, our goal will be  to explain what to do with the integral containing $e^{(A\mathbf{u},\mathbf{u})-\ovl{(A\mathbf{u},\mathbf{u})}}$ in the definition of twisted kernel in Section \ref{subsec:twist}. We wish to verify that the manipulations applied in that section (e.g. differentiation under the integral sign, integration by parts) are legitimate. In effect, we will formalize the details of the proof in Sec. \ref{subsec:twist}.

\begin{comment}
First, we can multiply the exponent by $e^{-\eps(\mathbf{u},\mathbf{u})}$ and define the twisted kernel as the limit when $\eps$ tends to zero. For nonzero $\eps$ all the computations with derivatives, integration by parts and dropping boundary terms are correct by Lemma~\ref{LemExponentIsGood} or analogous statements, but the result will not be zero, it will be $\eps$ times an integral. This expression will tend to zero when $\eps$ tends to zero. It remains to interchange partial derivative and limit, this is also a standard statement analogous to Lemma~\ref{LemDiffIntAbsolute}, but in this case it will require a lot more care than our statement and proof of Lemma~\ref{LemDiffIntAbsolute}. A more careful treatment was outlined \cite{Etingof2000}, and will be discussed in detail in \ref{sec:tempereddist}.
\end{comment}

Our starting point is the following integral: 
\[ \int d^m \bx d^m \bar{\bx} \left[ \prod_{i = 1}^m \abs{x_i}^s \right] e^{(\bx,A\bx)-\ovl{(\bx,A \bx)}} , \] where $(\bx,\by)=\sum_{i=1}^m x_iy_i$ is a bilinear form, not sesquilinear.
This integral is the primary object of interest in Sec. \ref{subsec:twist}, but a rigorous definition thereof is not immediately clear. Moreover, our manipulations of differentiating this integral and integrating by parts are not \textit{a priori} justified. In order to rectify this situation, we will need to define this integral using the notion of a Fourier transform of a tempered distribution. 

Note that the results of this section are not new: distrubutions and operations on them are known since the middle of twentieth century. We include proofs of these results for reader's convenience. It is possible that the properties of the integral used for Gaiotto kernel are described somewhere in the mathematical analysis literature, but the authors could not locate a source formalizing precisely the manipulations used in Sec. \ref{subsec:twist}. For this reason we (re)prove the results we need below.

We say that a smooth function $f\colon \RN^n\to\CN$ is \textit{rapidly decreasing} if for all $n\geq 0$ function $|x|^nf$ is bounded. Let $\mc{S}=\mc{S}(\RN^n)$ be the Schwartz space. It consists of smooth functions $f$ such that $f$ and all its derivatives are rapidly decreasing. The Schwartz space is a topological space with a family of seminorms $\|f\|_{\alpha,\beta}=\sup_{x\in\RN^n} |x^{\alpha}D^{\beta}f|$.
We further define $\mc{S}(\CN^n)$ via $\mc{S}(\RN^{2n})$. Define the \textit{space of tempered distributions} to be continuous dual $\mc{S}'$ of $\mc{S}$. 

The \textit{Fourier transform} $\mc{F}$ is an automorphism $\mc{F} : \mc{S} \to \mc{S}$ defined in the usual way. We can then define Fourier transform on $\mc{S}'$ as a transpose operator. We now introduce some convenient notation. For functions $f,g$ on $\RN^n$ define a bilinear pairing $(f, g)$ via
\[ (f,g) := \int_{\RN^n}d^n\bx f(\bx)g(\bx). \]
Again, this notation is slightly nonstandard, but it will be more convenient for us. We embed $L^p(\RN^n)$, $1\leq p\leq\infty$ into $\mc{S}'(\RN^n)$ via $f(g)=(f,g)$. In this way Fourier transform on $\mc{S}'$ extends Fourier transform $\mc{F}\colon L^1\to L^{\infty}$.

If Fourier transform of $f$ can be computed as a well-behaved conditionally convergent integral, the result will coincide with $\mc{F}f$:

\begin{prop}
\label{PropDecentConditionalFourierIsSchwartsFourier}
Let $f$ be a locally integrable function that belongs to $\mc{S}'(\RN^m)$, 
and let $K_1\subset K_2\subset\cdots$ be a sequence of compact sets with union $\RN^m$. Suppose that 
\[ I_n(\by) := \int_{K_n} d^m \bx \, f(\bx)e^{i(\bx,\by)} \]
is uniformly bounded in $n,\by$ and $\lim_{n \to \infty} I_n(\by)$ exists for all $\by$ and is uniform on compact sets. Then the function 
\[ g(\mb{y}) = \lim_{n \to \infty} I_n(\by)\] 
defined by this limit satisfies $g \in \mc{S}'$; moreover, $g = \mc{F}f$.
\end{prop}
\begin{proof}
We have \begin{multline}
\label{EqMcFHEqualsToLimit}
(\mc{F}f,h)=(f,\mc{F}h)=\int d^m\by \, f(\by)\mc{F}h(\by)=\\
\lim_{n \to \infty} \int_{K_n} d^m\by \, f(\by)\mc{F}h(\by)=\lim_{n\to\infty}\int_{K_n} d^m\by \, f(\by)\left[ \int_{\RN} d^m\bx \, e^{i(\bx,\by)}h(\bx) \right] =\\
\lim_{n \to\infty} \int_{\RN} d^m \bx \, h(\bx)\left[ \int_{K_n} d^m\by \, 
 e^{i(\bx,\by)} 
 f(\by) \right]
\end{multline}
In the last step, we used that $f$ is locally integrable to apply Fubini's theorem. Since $I_n(\by) := \int_{K_n} d^m \by \, e^{i(\bx, \by)}f(\by)$ is uniformly bounded, for any $\eps>0$ there exists a compact set $K_{\eps}$ such that the difference between 
\[\int_{\RN} d^m \bx \, h(\bx)\left[ \int_{K_n}  d^m\by  \, 
 e^{i(\bx,\by)}f(\by)\right] \]
and
\[\int_{K_\eps } d^m \bx \, h(\bx)\left[ \int_{K_n} d^m\by  \, e^{i(\bx,\by)}f(\by)\right] \]
is at most $\eps/2$ for any $n$. We may take $K_{\eps}$ so that it contains the ball with the radius $\frac{1}{\eps}$ centered at zero.

Since the limit of $\int_{K_n} d^m \bx \,  f(\bx)e^{i(\bx,\by)}$ is uniform on compact sets we have
\[\lim_{n \to \infty} \int_{K_{\eps}} d^m \bx \, h(\bx)\left[ \int_{K_n}  d^m\by  \, 
 e^{i(\bx,\by)}f(\by)\right]  = \int_{K_\eps } d^m \bx \, h(\bx)\left[ \lim_{n \to \infty}\int_{K_n} d^m\by  \, e^{i(\bx,\by)}f(\by)\right] \]
\[ =\int_{K_{\eps}} d^m \bx \, h(\bx)g(\bx).\]

It follows that for large enough $n$ the difference between 
\[\int_{\RN} d^m \bx \, h(\bx)\left[ \int_{K_n} d^m \by \, e^{i(\bx,\by)}f(\by) \right] \quad \text{and}  \quad \int_{K_{\eps}} d^m \bx \, h(\bx)g(\bx) \]
is at most $\eps$. Using~\eqref{EqMcFHEqualsToLimit}, we deduce that the difference between $(\mc{F}f,h)$ and $\int_{K_{\eps}} d^m \bx \, h(\bx)g(\bx)$ is at most $\eps$. 

Note that $g(\bx)$ is a bounded continuous function, hence it gives an element of $\mc{S}'$ and the integral $\int_{\RN^m}  d^m \bx \, h(\bx)g(\bx)$ can be computed as the limit $\lim_{\eps\to 0}\int_{K_{\eps}}d^m\bx h(\bx)g(\bx)=(\mc{F}f,h)$, hence $g=\mc{F}f$.
\end{proof}

Equipped with a definition of the Fourier transform on tempered distributions, we may now define the integral over a tempered distribution: 

\begin{defn}
Let $f$ be a tempered distribution such that $\mc{F}f$ is represented by a function continuous at zero. Then we define the integral of $f$ as  
\[ \int_{\RN}dx \, f(x) := \mc{F}f(0). \]
\end{defn}
%For example, we have $(\mc{F}e^{ix^2})(\ksi)=(1+i)\sqrt{\frac{\pi}2}e^{-i\pi^2\ksi^2}$, so we define $\int e^{ix^2}$ to be $(1+i)\sqrt{\frac{\pi}{2}}$. 

Thus, in order to define an integral 
\[ \int d^m \bx d^m \bar{\bx} \left[ \prod_{i = 1}^m \abs{x_i}^{s_i} \right]  e^{(\bx,A \bx)-\ovl{(\bx,A\bx)}} \] 
in this sense, we should compute 
\[ \mc{F}\left( \left[ \prod_{i = 1}^m \abs{x_i}^{s_i} \right]  e^{(\bx,A \bx)-\ovl{(\bx,A\bx)}} \right). \]
We will break up this computation into several steps. First we will compute $\mc{F}e^{(\bx,A\bx)-\ovl{(\bx,A\bx)}}$ in Lemma \ref{LemFourierTransformOfQuadraticExponent}. The function $\abs{x_i}^{s_i}$ gives an element of $\mc{S}'$ only for $\Re s_i>-2$, but the corresponding distribution can be meromorphically continued to $\CN\setminus -2\ZN_{>0}$, and its Fourier transform is known. A usual Fourier transform of a product of \textit{functions} is a convolution of Fourier transforms under some conditions. In our case, for certain values of $s_i$ we can also make sense of the integral in the definition of convolution and the result is a smooth function. We prove that it gives a Fourier transform of the product and then use meromorphic continuation to prove that Fourier transform is always given by a smooth function, hence the integral is well-defined. Then we check that the integral satisfies the properties we need: we can differentiate under integral sign and in certain cases we can integrate by parts with no boundary term.

To simplify notation, define 
\begin{equation} \label{eqn:EAxandMx}
E(\bx) = E(A,\bx):=e^{(\bx,A\bx)-\ovl{(\bx,A\bx)}}, \quad M(\bx)=M(s_1,\ldots,s_n,\bx) := \prod_{i=1}^n \abs{x_i}^{s_i}.
\end{equation}

\begin{lem}
\label{LemFourierTransformOfQuadraticExponent}
There exist fixed complex numbers $C_1,C_2$ such that the following hold:
\begin{enumerate}
\item
$\mc{F}e^{ix^2}=C_1e^{-i\frac{1}{4}y^2}$;
\item
Suppose that $A$ is invertible complex matrix, $\bz\in\CN^n$. Then \[\mc{F}e^{(\bz,A\bz)-\ovl{(\bz,A\bz)}}=\frac{C_2}{\abs{\det A}}e^{\ovl{(\bw,A^{-1}\bw)}-(\bw,A^{-1}\bw)}.\]
The second proposition can be rewritten as $\mc{F}E(A,\bx)=\frac{C_2}{\abs{\det A}}E(-A^{-1},\bx)$.
\end{enumerate}
\end{lem}
\begin{proof}

Let us first prove the first statement. It is well-known that $\int dx \, e^{ix^2}$ is conditionally convergent: We have
\[ \int_{a}^{\infty} dx \, e^{ix^2} \to 0, \quad \int_{-\infty}^a dx\, e^{ix^2} \to 0 \]
as $a \to \infty$. It is also not hard to see that $\int_a^b e^{ix^2}dx$ is bounded for all $a,b$. Now, note that
\[ e^{ix^2+ixy}=e^{-\frac{iy^2}{4}}e^{i(x+\frac{y}{2})^2} . \] 
We deduce that $\int_{-n}^n e^{ix^2+ixy}$ is uniformly bounded. Moreover, for a compact set $[-n,n]$, 
\[ \int_{-n}^n dx\, e^{ix^2+ixy} =\int_{-n}^n dx\,  e^{-\frac{iy^2}{4}}e^{i(x+\frac{y}{2})^2}=e^{-\frac{iy^2}{4}}\int_{-n+\frac{y}{2}}^{n+\frac{y}{2}}dx\, e^{ix^2},  \]

Let $S\subset [-N,N]$ be a compact set, $y\in S$. Then 
\[ \abs{ \int_{-n}^n dx\, e^{ix^2+ixy}-e^{-\frac{iy^2}{4}} }=e^{-\frac{iy^2}{4}} \abs{ \int_{-\infty}^{-n+\frac{y}{2}} dx\, e^{ix^2}+\int_{n+\frac{y}{2}}^{\infty} dx\, e^{ix^2} } \leq 2\abs{\int_{n-N}^\infty dx\, e^{ix^2}}.\]
The right hand side tends to zero as $n$ tends to infinity. Hence the limit 
\[ \lim_{n \to \infty} \int_{-n}^n dx\,  e^{ix^2+ixy}=e^{-\frac{iy^2}{4}} \] 
is uniform on compact sets. It follows that $f(x)=e^{i(\bx, \bx)}$ and $K_n=[-n,n]^m$ satisfy the conditions of Proposition~\ref{PropDecentConditionalFourierIsSchwartsFourier} and the limit is $g(\by)$ is a constant times $e^{-i\frac{1}{4}(\by,\by)}$.

For the second statement, make the coordinate change $\bz \to \sqrt{A}^{-1}\bz$. It then suffices to show the result for $A$ equal to the identity. Let $\bz=(z_1,\ldots,z_n)$, $z_j=x_j+iy_j$. In this case, 
\[ (\bz,\bz)-\ovl{(\bz,\bz)}=2i\sum_{j = 1}^n x_jy_j=\frac{i}{2}\sum_{j = 1}^n (x_j+y_j)^2-(x_j-y_j)^2. \]
Now, let $\bw = (w_1, \cdots, w_n)$, where $w_j = a_j + i b_j$. Reasoning similarly to the proof of the first statement, we find that
\[
\mc{F}e^{\frac{i}{2}\sum_{j = 1}^n (x_j+y_j)^2-(x_j-y_j)^2}
 =C_2 e^{-\frac{i}{2}\sum_{j = 1}^n (a_j+b_j)^2-(a_j-b_j)^2} \]
 \[ = C_2 e^{-2i\sum a_jb_j}=C_2e^{\ovl{(\bw,\bw)}-(\bw,\bw)}. \]
for some absolute constant $C_2$. This finishes the proof.
\end{proof}

\begin{lem}
\label{LemAbsUSIsMeromorphicDistribution}
Let $\Re s>-2$, $\mc{S}=\mc{S}(\CN)$, $f\in\mc{S}$. Then $(\abs{u}^s,f)$ is well-defined and gives meromorphic function in $s$, holomorphic over $s \in \CN \setminus 2\ZN$.
\end{lem}
\begin{proof}
Since $\Re s>-2$, the function $\abs{u}^s$ is locally integrable and has at most polynomial growth at infinity. Thus, the function $\abs{u}^s$ belongs in $\mc{S}'$. It remains to show that $(\abs{u}^s,f)$ is meromorphic with poles at $2\mathbb{Z}_{< 0}$. Since $f$ decays rapidly, we note that for any $V \subset \mathbb{C}$, 
\[ \int_{\CN} du d\ovl{u} \, \max_{s \in V} \abs{ \p_s^j \left(  \abs{u}^s f \right) } = \int_{\CN} du d\ovl{u} \, \max_{s \in V} \abs{  \abs{u}^s \log^j \abs{u} f } . \]
The integral on the right hand side is always absolutely convergent for $\Re s > -2$. We conclude by lemma~\ref{LemDiffIntAbsolute} that $(\abs{u}^s,f)$ is infinitely differentiable in $s$ (and thus holomorphic) when $\Re s> -2$.

Now, notice that 
\[ \Delta\abs{u}^s=\partial_u\partial_{\ovl{u}}(u\ovl{u})^{\frac{s}{2}}=\frac{s^2}{4}\abs{u}^{s-2}. \] 
Suppose that $\Re s>0$. Let $\eps>0$, $K=\{u\mid \eps<\abs{u}<\eps^{-1}\}$. Let us transform the following integral using Green's theorem:
\[ \int_K \Delta(\abs{u}^s)f dxdy. \] 
The boundary term on the circle $\abs{u}=\eps^{-1}$ clearly tends to zero, so we look at boundary term on the circle $\abs{u}=\eps$. For any two smooth functions $g,h$ on $\RN^2\setminus\{0,0\}$ with compact support for large enough $\eps$ we have \[\int_K (\Delta g) h =\int_K (g''_{xx}+g''_{yy})h dxdy=\int_{\abs{u}=\eps}(hg'_x dy-hg'_y dx)-\int_K (g'_xh'_x+g'_yh'_y)dxdy.\]

When $g=f$, $h=\abs{u}^s$, both $hg'_x$ and $hg'_y$ are continuous at zero, hence the boundary term goes to zero.

When $g =\abs{u}^s$ and $h = f$ the boundary term tends to zero when $\eps$ tends to zero. More explicitly, writing $g=\abs{u}^s=(x^2+y^2)^{\frac{s}{2}}$, we have 
\[ g'_x=sx\abs{u}^{s-2}, \quad f'_y=sy\abs{u}^{s-2} .\] 
Hence,
\[ hg'_x dy-hg'_ydx=sh\abs{u}^{s-2}(xdy-ydx) . \] 
If we parametrize the boundary circle using $x=r\cos\phi$, $y=r\sin\phi$ we get $xdy-ydx=r^2d\phi$. Note that $r^2\abs{u}^{s-2}=\abs{u}^s$. The function $h\abs{u}^{s-2}$ tends to zero when $\abs{u}$ tends to zero, hence this boundary term tends to zero when $\eps \to 0$. 

Thus, using Green's theorem twice allows us to write (for $\Re s>0$)
\[(\abs{u}^{s-2},f)=\frac{4}{s^2}(\Delta\abs{u}^s,f)=\frac{4}{s^2}(\abs{u}^s,\Delta f).\] 
This allows us to define $(\abs{u}^s,f)$ for $s$ such that $\Re s>-4$, $s\neq -2$. It also proves that $(\abs{u}^s,f)$ is holomorphic for $\Re s>-4$ except at $s = -2$, where it has a pole of order $2$. Continuing in this way, we see that $(\abs{u}^s,f)$ continues to a meromorphic function with poles at $-2, -4, \ldots$ and so on. 
\end{proof}

\begin{lem} \label{LemFourierTransformOfus}
Let $-2<\Re s<0$, $\abs{u}^s\in \mc{S}'(\CN)$. Then the Fourier transform $\mc{F}\abs{u}^s\in \mc{S}'$ is given by locally integrable function $C(s)\abs{u}^{-s-2}$, where 
\[ C(s) = 2^{-s} \pi \frac{ \Gamma \left(-\frac{s}{2} \right)  }{ \Gamma \left(  \frac{2+s}{2} \right) }. \]
The same formula holds if we meromorphically continue $s$ onto $s \in \CN\setminus 2\ZN$.
\end{lem}

\begin{proof}
This formula is quoted in Gelfand-Shilov, p. 363, and the relevant function $r^{\lambda}$ is defined in Sec. 1.3.9, p. 71 of \cite{gelfand64}. It follows from the computations on p. 73 that our definition is the same; the proof is in Sec. 2.3.3, p. 192.
\end{proof}

Now we explicitly construct a Fourier transform of
\[ f(\bx) = \left[ \prod_{i=1}^n \abs{x_i}^{s_i} \right] e^{(\bx,A\bx)-\ovl{(\bx,A\bx)}}   \] 
in the case when all $s_i$ satisfy $\Re s_i<-2$.

How would we go about constructing this Fourier transform? A natural answer would be a convolution of 
\[ \prod C(s_i)\abs{y_i}^{-2-s_i} \quad \text{and} \quad \frac{C_2}{\abs{\det A}}e^{\ovl{(\by,A^{-1}\by)}-(\by,A^{-1}\by)} \] 
in some sense. Below, we will define this convolution in such a way to get function smooth in $\by,A$ and meromorphic in the $s_i$.

Let $f$ be a smooth function on $\CN = \RN^2$ bounded at infinity. Choose a partition of unity $1=\rho_1+\rho_2$, where $\rho_1$ has compact support and the support of $\rho_2$ does not contain zero. For $\Re s<-2$, we can define the integral of $\abs{u}^s f(u)$ over $\mathbb{C}$ as follows. Let $f_i=f\rho_i$. Then $\int_{\CN} du d\ovl{u} \, |u|^s f_2(u) $ is convergent. We now deduce the following fact from the proof of Lemma~\ref{LemAbsUSIsMeromorphicDistribution}: When the support of $f_1$ does not intersect with zero for any positive integer $m$ we have 
\begin{equation} \label{eqn:suppf1doesnotintersect0}
    \int_{\CN} du d\ovl{u} \, f_1(x)\abs{u}^s=c(s,m)\int_{\CN} du d\ovl{u} \, (\Delta^m f(u))\abs{u}^{s+2m} ,
\end{equation}
\[  c(s,m)=\frac{4^m}{s^2(s+2)^2\cdots (s+2m-2)^2} . \]
When the support of $f$ contains zero, the left-hand side is not well-defined because $\Re s < -2$, but we can use the right-hand side as a definition:
\begin{equation} \label{eqn:suppf1doesintersect0}
    \int_{\CN} du d\ovl{u} \,  f_1(u)\abs{u}^s := c(s,m)\int_{\CN}  du d\ovl{u} \, (\Delta^m f_1 (u) )\abs{u}^{s+2m} ,
\end{equation}
where we choose $m$ so that $\Re s+2m>-2$. It follows from the computations with Green's theorem in the proof of Lemma~\ref{LemAbsUSIsMeromorphicDistribution} that the result does not depend on the choice of $m$. We also see that the result does not depend on the choice of $\rho_1,\rho_2$.

We can similarly define the integral over $\CN^n$ of $f(\bx)\prod_{i = 1}^n \abs{x_i}^{s_i}$ when $\Re s_i<-2$ for all $i$: Consider the sum 
\[ \sum_{k_i\in\{1,2\}}f(\bx)\prod \rho_{k_i}(x_i)\abs{x_i}^{s_i}, \] 
and for each term use \eqref{eqn:suppf1doesintersect0} integrate over the pieces containing a factor of $\rho_1(x_i)$. 

\begin{prop} \label{PropDefineg}
Suppose that $s_1, \cdots, s_n \in \CN$ such that $\Re s_i>0$. Let 
\[ g(\by)=\int_{\CN^n} d^n \bx d^n \ovl{\bx}  \,  E( -A^{-1},\bx + \by ) \prod_{i = 1}^n \abs{x_i}^{-2-s_i} , \] 

where we can use the construction above because $E(-A^{-1},\bx + \by)$ is smooth and has absolute value one. Then $g$ is smooth in $\by$, has $\lfloor\tfrac12 \min_i\Re s_i\rfloor-10$ partial derivatives in $A$, is holomorphic in the $s_i \in \CN \setminus 2\ZN$. The partial derivatives in $A$ belong to $\mc{S}'$. Furthermore, 
\[ \mc{F} \left[ e^{(\bx,A\bx)-\ovl{(\bx,A\bx)}}\prod_{i = 1}^n \abs{x_i}^{s_i} \right] = \frac{C_2}{\abs{\det A}} \left( \prod_{i = 1}^n C(s_i) \right) g (\by) . \]
 
\end{prop}
\begin{proof}
Write partitions of unity $1 = \rho_1(x_i) + \rho_2(x_i)$ as before, where $\rho_1$ has compact support and $0 \notin \mathrm{supp} (\rho_2)$. Then, we may write 
\begin{multline*}
    g(\by)=\sum_{k_i\in{1,2}}\prod_{i: k_i=1}c(-2-s_i,m)\int d^m \bx d^m \ovl{\bx} \\ \,  \prod_{i = 1}^n \abs{x_i}^{-2-s_i} \prod_{i \colon k_i = 1} \abs{x_i}^{2m} \bigg(\prod_{i\colon k_i=1} \Delta_{x_i}^m \bigg)  \left[ f(\bx+\by) \prod_{i = 1}^n \rho_{k_i}(x_i)  \right] 
\end{multline*} 

where we take $m$ large enough. We can apply Lemma~\ref{LemDiffIntAbsolute} for $\p_{s_i} g$: if $k_i=2$ the integral at infinity converges, while if $k_i=1$ the integral at zero converges if we take $m$ large enough. Note that $c(-2-s_i,m)=\frac{1}{(-2-s_i)^2(-s_i)^2\cdots(2m-4-s_i)^2}$ is holomorphic on $\CN\setminus 2\ZN$. Hence $g$ has a complex derivative with respect to $s_i \notin \CN \setminus 2\ZN$, so $g$ is holomorphic in this region. We also see that $g$ has at most polynomial growth: the integrals at infinity are bounded, the integrals are bounded by constant times $\max(1,\abs{y}^{2nm})$.

Let us now try to differentiate $g$ by $y_i$ or $a_{ij}$. When we attempt to do this, we could possibly obtain a divergence at infinity. When taking a derivative with respect to $a_{ij}$, we take at least $\lfloor\tfrac12 \min_i\Re s_i\rfloor-10$ derivatives and the integral will still converge at infinity. When differentiating we get terms in $y$ of polynomial growth, hence the partial derivatives of $g$ also belong to $\mc{S}'$.

Let us now show that $g(\by)$ is smooth in $\by$. To do this, replace $\rho_i(x_i)$ to $\rho_i(x_i+y_i-y_i^*)$: in a neighborhood of $y_i^*$. The functions $\rho_2$ are chosen such that they still have support that does not intersect with zero. We therefore find
\begin{multline*}
    g(\by)=\sum_{k_i\in{1,2}}\prod_{i: k_i=1}c(-2-s_i,m)\int d^m \bx d^m \ovl{\bx} \\ \,  \prod_{i = 1}^n \abs{x_i}^{-2-s_i} \prod_{i \colon k_i = 1} \abs{x_i}^{2m} \bigg( \prod_{i\colon k_i=1} \Delta_{x_i}^m \bigg) \left[ f(\bx+\by) \prod_{i = 1}^n \rho_{k_i}(x_i+y_i-y_i^*)  \right] 
\end{multline*}
Shifting the integration variable $\bx \to \bx - \by$, we get
\begin{multline*}
    g(\by)=\sum_{k_i\in{1,2}}\prod_{i\colon k_i=1}c(-2-s_i,m)\int d^m \bx d^m \ovl{\bx} \\ \prod_{i = 1}^n \abs{x_i-y_i}^{-2-s_i} \prod_{i\colon k_i=1} \abs{x_i-y_i}^{2m} \bigg( \prod_{i\colon k_i=1}  \Delta_{x_i}^m \bigg) \left[ f(\bx)\prod_{i = 1}^n \rho_{k_i}(x_i-y_i^*)  \right] 
\end{multline*}

The above expression has $2m-s_i-4$ partial derivatives in $y_i$: There are no issues at infinity, and we get divergence at zero only when $\abs{x_i-y_i}$ has negative degree. Since we can take any large enough $m$ this proves that $g$ is smooth in $y_i$. 

Now, take $\bx \to -\bx$ in the integral above and note that $f(\bx-\by)=f(\by-\bx)$. We find that 
\[ g(\bx)=\int d^m \bx d^m \ovl{\bx} \, f(\bx-\by) \prod\abs{x_i}^{-2-s_i} . \]
For $h\in \mc{S}$, we have
\begin{multline*}
(g,h)=\sum_{k_i\in \{1,2\}}\prod_{i\colon k_i=1}c(s_i,m)  \int d^m \by d^m \ovl{\by} \int d^m \bx d^m  \ovl{\bx} \\
\prod_{i = 1}^n \abs{x_i}^{-2-s_i} \prod_{i\colon k_i=1}\abs{x_i}^{2m} \bigg( \prod_{i\colon k_i=1}\Delta_{x_i}^m \bigg) \left[  f(\bx-\by) \prod_{i = 1}^n \rho_{k_i}(x_i) \right] h(\by).
\end{multline*}
Integrating over $\by$, we get
\begin{multline*} \sum_{k_i\in \{1,2\}}\prod_{i\colon k_i=1}c(s_i,m)  \int d^m \by d^m \ovl{\by} \int d^m \bx d^m  \ovl{\bx} \\
\prod_{i = 1}^n \abs{x_i}^{-2-s_i} \prod_{i\colon k_i=1}\abs{x_i}^{2m} \bigg( \prod_{i\colon k_i=1}\Delta_{x_i}^m \bigg) \left[  f(\bx-\by) \prod_{i = 1}^n \rho_{k_i}(x_i) \right] h(\by)
\end{multline*}
\begin{multline*}
= \sum_{k_i\in \{1,2\}}\prod_{i\colon k_i=1}c(s_i,m)  \int d^m \bx d^m  \ovl{\bx} \\
\prod_{i = 1}^n \abs{x_i}^{-2-s_i} \prod_{i\colon k_i=1}\abs{x_i}^{2m}\bigg(\prod_{i\colon k_i=1}\Delta_{x_i}^m \bigg) \left[  (f*h)(\bx) \prod_{i = 1}^n \rho_{k_i}(x_i)\right].
\end{multline*}

We now note that $f*h=\mc{F}^{-1}(\mc{F}f\mc{F}h)$ is a rapidly decreasing function:

\begin{lem}
Suppose that $f\in\mc{S}$ is a function such that $f,\mc{F}f$ are bounded. Then for any $h\in \mc{S}$ we have $f*h\in \mc{S}$, $f*h=\mc{F}^{-1}(\mc{F}f\mc{F}h)$.
\end{lem}
\begin{proof}
Function $f*h(\by)=\int d^n \bx \, f(x)h(y-x)$ is defined by an absolutely convergent integral. It is also bounded, hence it belongs to $\mc{S}'$. For $g\in\mc{S}$ we use Fubini theorem to obtain
\begin{multline*}
(f*h,\mc{F}g)=\int d^n \by d^n \ovl{\by} \, f*h(\by)(\mc{F}g)(\by)= \int d^n\bx d^n \ovl{\bx} d^n\by d^n \ovl{\by} d^n\bz d^n \ovl{\bz} \,  f(\bx)h(\by-\bx)e^{i(\by,\bz)}g(\bz)\\
= \int d^n\bx d^n \ovl{\bx} d^n\bz d^n \ovl{\bz} \, f(\bx)\mc{F}h(\bz)e^{i(\bx,\bz)}g(\bz) = \int 
d^n \bx d^n \ovl{\bx} \, f(\bx)\mc{F}(\mc{F}h\cdot g)(\bx)\\
= (f,\mc{F}(\mc{F}h\cdot g))=(\mc{F}f,\mc{F}h\cdot g)=(\mc{F}f\mc{F}h,g),
\end{multline*}
hence $\mc{F}(f*h)=\mc{F}f\mc{F}h$ as elements of $\mc{S}'$. Since $\mc{F}f$ is bounded and $\mc{F}h\in \mc{S}'$, their product belongs to $\mc{S}$, hence $\mc{F}(f*h)$ also belongs to $\mc{S}$. Function $f*h$ and $\mc{F}^{-1}\mc{F}h$ give the same element of $\mc{S}'$, hence $f*h=\mc{F}^{-1}\mc{F}(f*h)=\mc{F}^{-1}(\mc{F}f\mc{f}h)$ belongs to $\mc{S}$.
\end{proof}

Hence for each $(k_1,\ldots,k_n)$ we can integrate by parts using $\Delta_j^m$ for all $j$ such that $k_j=2$ to get
\begin{multline*}
\sum_{k_i\in \{1,2\}}\prod_{i\colon k_i=1}c(s_i,m)  \int d^m \bx d^m  \ovl{\bx} \\
\prod_{i = 1}^n \abs{x_i}^{-2-s_i} \prod_{i\colon k_i=1}\abs{x_i}^{2m}\bigg(\prod_{i\colon k_i=1}\Delta_{x_i}^m \bigg) \left[  (f*h)(\bx) \prod_{i = 1}^n \rho_{k_i}(x_i)\right]
\end{multline*}
\begin{multline*}
= \sum_{k_i\in \{1,2\}}\prod_{i = 1}^nc(s_i,m)  \int d^m \bx d^m  \ovl{\bx} \\
\prod_{i = 1}^n \abs{x_i}^{-2-s_i} \prod_{i 
= 1}^n\abs{x_i}^{2m}\bigg(\prod_{i = 1}^n\Delta_{x_i}^m \bigg) \left[  (f*h)(\bx) \prod_{i = 1}^n \rho_{k_i}(x_i)\right]
\end{multline*}

Now we can sum over all $(k_1,\ldots,k_n)$ to get 

\begin{multline*}
\sum_{k_i\in \{1,2\}}\prod_{i = 1}^n c(s_i,m)  \int d^m \bx d^m  \ovl{\bx} \\
\prod_{i = 1}^n \abs{x_i}^{-2-s_i} \prod_{i = 1}^n \abs{x_i}^{2m}\bigg(\prod_{i}\Delta_{x_i}^m \bigg) \left[  (f*h)(\bx) \prod_{i = 1}^n \rho_{k_i}(x_i)\right]=\\
\prod_{i =1}^n c(s_i,m)\int d^m \bx d^m  \ovl{\bx}\prod_{i = 1}^n \abs{x_i}^{-2-s_i} \prod_{i =1}^n\abs{x_i}^{2m}\bigg(\prod_{i = 1}^n \Delta_{x_i}^m \bigg)(f*h)(\bx)=\\
\prod_{i = 1}^n c(s_i,m) \left( \prod_{i = 1}^n \abs{x_i}^{2m-2-s_i},(f*h)(\bx) \right) =\left(\prod_{i = 1}^n \abs{x_i}^{-2-s_i}, f*h \right)
\end{multline*}

\begin{comment}
Fix $j$ such that $k_j=2$. Fix all $x_i,y_i$ except $x_j,y_j$, and treat $f, h$ as functions of $x_j, y_j$ alone. Then we have \begin{multline*}\int dx_j d\bar{x}_j dy_j d\bar{y}_j \, \abs{x_j}^{-2-s_j} \rho_2(x_j)f(x_j-y_j) h(y_j)= \int  dx_j d\bar{x}_j \,  \abs{x_j}^{-2-s_j} \rho_2(x_j)(f*h)(x_j )dx_j\\
=(\abs{x_j}^{-2-s_j},\rho_2(x_j)(f*h)(x_j))=c(-2-s_j,m)(\abs{x_j}^{2m-2-s_j},\Delta_j^m\left[ \rho_2(x_j)(f*h)(x_j) \right] ).
\end{multline*}
Here we used that $f*h$ is rapidly decreasing as a Fourier transform of a rapidly decreasing function.
\end{comment}
Hence, we have proven that 
\begin{equation}
\label{EqGHisFConvH}
(g,h)= \left(\prod_{i=1}^n \abs{x_i}^{-2-s_i},f*h \right).
\end{equation}
Recall that $M(\bx) = \abs{x_1}^{s_1} \cdots \abs{x_n}^{s_n}$. We now compute $\mc{F}\left[ E(A,\bx)M(\bx) \right]$. We have 
\[ \left( \mc{F}(E(A,\bx) M(
\bx) ),h \right)= \left(E(A,\bx) M(\bx) ,\mc{F}h \right)=(M(\bx) ,E(A, \bx)\mc{F}h) .\]
By Lemma \ref{LemFourierTransformOfQuadraticExponent}, we find that 
\[  (M(\bx) ,E(A,x)\mc{F}h)  = \frac{C_2}{\abs{\det A}} \left( M(\bx),\mc{F}(E(A^{-1},x)*h) \right)= \frac{C_2}{\abs{\det A}} (M(\bx) ,\mc{F}(f*h)) . \]
Finally, by Lemma \ref{LemFourierTransformOfus}, we find that  
\begin{multline*}
\frac{C_2}{\abs{\det A}} (M(\bx) ,\mc{F}(f*h)) = \frac{C_2}{\abs{\det A}} \left(\mc{F}\left[\prod_{i = 1}^n\abs{x_i}^{s_i} \right],f*h \right) \\ = \frac{C_2}{\abs{\det A}}\prod_{i = 1}^n C(s_i)
\left(\prod_{i = 1}^n \abs{x_i}^{-2-s_i},f*h \right) .     
\end{multline*} 
Comparing this with~\eqref{EqGHisFConvH}, we get 
\[ \mc{F}\left[ E(A,\bx) M(\bx) \right] =  \frac{C_2}{\abs{\det A}} \left( \prod_{i = 1}^n C(s_i) \right) g(\by) .
 \]%\begin{lem}
%$f\mc{F}h=\mc{F}(\mc{F}^{-1}f*h)$ for $h\in\mc{S}$, $f=e^{(x,Ax)-\ovl{(x,Ax)}}$.
%\end{lem}
%\begin{proof}
%Since $\mc{F}^{-1}f$ is also smooth and bounded expression $\mc{F}^{-1}f*h$ is well-defined.    
%\end{proof}
\end{proof}

\begin{prop}
Let $s_1,\ldots,s_n\in\CN\setminus 2\ZN$, and recall that $E(A,\bx)M(\bx) $ defines an element of $\mc{S'}$ as above.  Then $g(\by) := \mc{F}(E(A,\bx) M(\bx) )$ is given by a function smooth in $\by$ and $A$ and holomorphic over $s_i \in \CN\setminus 2\ZN$.
\end{prop}
\begin{proof}
We have $(\mc{F}(E(A,\bx)M(\bx),h)=(M(\bx),E(A,x)\mc{F}h)$. Thus, 
\[ ( M(\bx) ,E(A,\bx)\mc{F}h ) =\prod c(s_i,m)\left( \prod_{i = 1}^n \abs{x_i}^{s_i+2m},\prod\Delta_i^m(E(A,\bx)\mc{F}h) \right).\]
%this equals to \[\prod c(s_i,m)(\prod\abs{x_i}^{s_i+2m},f*\mc{F}(\prod x_i^m h))=\prod c(s_i,m)(\mc{F}(e^{(x,Ax)-\ovl{(x,Ax)}}\prod\abs{x_i}^{s_i+2m},\prod x_i^m h).\]

Note that for any partial differential operator $D$ with constant coefficients we have $D E(\bx)=P_D(\bx, \ovl{\bx}) E(\bx)$, where $P_D$ is a polynomial in $x,\ovl{x}$. Hence for any partial differential operator $D$ with constant coefficients we have \[D(E(\bx)\mc{F}h)=\sum_{D_1,D_2}D_1E(\bx)D_2\mc{F}h= E(\bx) \sum_{P_1,D_2} P_1(\bx, \ovl{\bx})  D_2\mc{F}h.\]

Here sum comes from the formula for the derivative of a product, $P_1=P_{D_1}$ is a polynomial such that $D_1E(\bx)=P_1 E(\bx)$. Hence $D(E(x)\mc{F}h)=E(x)\tilde{D}\mc{F}h$, where $\tilde{D}$ is a partial differential operator with polynomial coefficients. Since Fourier transform interchanges $x$ with $\partial_x$ (up to a constant factor), we get $\tilde{D}\mc{F}h=\mc{F}(\bold{D} h)$, where $\bold{D}$ is also a partial differential operator with polynomial coefficients. When $D=\prod_{i = 1}^n \Delta_i^m$, denote the corresponding $\bold{D}$ by $\bold{D}_m$. 

Hence, defining
\[ M_m(\bx) : = \prod_{i = 1}^n 
 \abs{x_i}^{s_i+2m},  \]
we find that 
\[\prod_{i = 1}^n c(s_i,m)\left(M_m(\bx) ,\prod_{i = 1}^n \Delta_i^m(E(A,x)\mc{F}h) \right)=\prod_{i = 1}^n c(s_i,m)(M_m(\bx),E(A,x)\mc{F}\bold{D}_mh).\]

We already know that 
\[ (M_m(\bx),E(A,\bx)\mc{F}h)=(g_m(\bx),h(\bx)), \]
where for $m$ sufficiently large we may write
\[ g_m(\by) = \int_{\CN^n} d^n \bx d^n \ovl{\bx}  \,  E( -A^{-1},\bx + \by ) \prod_{i = 1}^n \abs{x_i}^{-2-s_i - 2m} \]
as per Proposition \ref{PropDefineg}. Note that $g_m$ is smooth in $\by$, holomorophic on $s_i\in \CN\setminus 2\ZN$ and has at least $m-c$ partial derivatives in $a_{ij}$ for some constant $c > 0$ that does not depend on $m$. We thus find
\begin{multline*}
     \left[ \prod_{i = 1}^n c(s_i,m) \right] (M_m(\bx),E(A,x)\mc{F}\bold{D}_mh) \\ = \left[ \prod_{i = 1}^n c(s_i,m) \right] (g_m(\by),\bold{D}_m h(\by)) = \left[\prod_{i = 1}^n c(s_i,m)  \right] (\bold{D}_m^*g_m(\by),h(\by)),
\end{multline*}
hence 
\[\mc{F}(E(A,\bx)M(\bx) )= \left[ \prod_{i = 1}^n c(s_i,m) \right] \bold{D}_m^*g_m(\by)\]
for $m$ large enough. The right-hand side is holomorphic for $s_i  \in \CN\setminus 2\ZN$ and smooth in $\by$ and also has $m-c$ partial derivaties in $a_{ij}$. Since the left-hand side does not depend on $m$ we deduce that is smooth in $a_{ij}$.
%Hence $\mc{F}(e^{(x,Ax)-\ovl{(x,Ax)}}\prod \abs{x_i}^{s_i})=\prod c(s_i,m)x_i^m \mc{F}(e^{(x,Ax)-\ovl{(x,Ax)}})\prod\abs{x_i}^{s_i+2m})$. So it is a smooth in $y_i$, meromorphic in $s_i$ and has as many partial derivatives in $A$ as we want.
\end{proof}

%\begin{prop}
%Let $u=(u_1,\ldots,u_n)\in \CN$, $s_i\notin 2\ZN$, $A$ is an $n\times n$ complex matrix with $\det A\neq 0$, 
%\[ h(u)=\prodl_{i=1}^n \abs{u_i}^{s_i}e^{(u,Au)-\ovl{(u,Au)}} .\] 
%Using Lemma~\ref{LemAbsUSIsMeromorphicDistribution}, we can consider $h$ as a tempered distribution. Then $\mc{F}h$ is a distribution given by the function $g\approx f*\prod \abs{u_i}^{-2-s_i}$ defined above. \color{red} What is $f$? \color{black}
%\end{prop}
%\begin{proof}
%Let $g\in \mc{S}$. Note that $f_2g$ also belongs to $\mc{S}$. Hence we have $(\mc{F}f,g)=(f,\mc{F}g)=(f_1f_2,\mc{F}g)=(f_1,f_2\mc{F}g)$. The right-hand side is meromorphic in $s_1,\ldots,s_n$ by Lemma~\ref{LemAbsUSIsMeromorphicDistribution}, hence $(\mc{F}f,g)$ is also a meromorphic function.

%Similarly we prove that $(\mc{F}f,g)$ is (smooth with respect to $a_{ij})$.

%We write $f_2\mc{F}g=\mc{F}(\mc{F}^{-1}f_2*g)$ (this is another lemma).

%Hence $(f_1,f_2\mc{F}g)=(f_1,\mc{F}(\mc{F}^{-1}f_2*g))=(\mc{F}f_1,\mc{F}^{-1}f_2*g)=(\mc{F}f_1*\mc{F}f_2,g)$.

%It is enough to prove this when $-2<\Re s_i<0$ since both $\mc{F}h$ and $g$ are holomorphic distributions on $\CN\setminus 2\ZN$.

%We have $g(y)=\sum\int \mc{F}f\prod \rho_{k_i}(x_i)\abs{x_i}^{-2-s_i}$

%We have $(\mc{F}(gf),h)=(g_1g_2,\mc{F}h)=(gf,\mc{F}(h))=(gf,\rho_1\mc{F}(h)+\rho_2\mc{F}(h))$
%\end{proof}
\begin{prop} \label{prop:diffoponfourier}
Let $D_A$ be a differential operator in $a_{ij}$. Then 
\[ D_A(\mc{F}(E(A,\bx) M(\bx) ))=\mc{F}(D_A(E(A,\bx)) M(\bx)). \]
\end{prop}
\begin{proof}
It is enough to check that both sides give the same element of $\mc{S}'$. It is not hard to see that for any $A^*$ there exists a neighborhood $U$ such that 
\[ \max_{A\in U} D_A\mc{F}(E(A,\bx) M(\bx) ) \]
is still an element of $\mc{S}'$. Indeed, $g(\by) = \mc{F}(E(A,\bx) M(\bx) ) $ is smooth in $A$ and $\by$, so we may safely take $D_A g(\by)$. We find
\[ D_A g(\by) = P(\by, \ovl{\by}) g(\by) \]
for some polynomial $P$; integrating $D_A g(\by)$ against any test function $h \in \mc{S}$ thus yields an absolutely convergent integral. 

Thus, we use Lemma~\ref{LemDiffIntAbsolute} to get 
\[ \left(D_A(\mc{F}(E(A,\bx)M(\bx))),h \right)=D_A(\mc{F}(E(A,\bx) M(\bx)),h)=D_A(E(A,\bx)M(\bx),\mc{F}h).\]
We then find 
\[D_A(E(A,\bx)M(\bx),\mc{F}h)=(D_A(E(A,\bx))M(\bx),\mc{F}h).\]
The proposition follows.
\end{proof}
\begin{cor}
Let $K=g(0)$ be the twisted Gaiotto kernel, $D_A$ be as above. Then $D_A K$ can be computed as an integral in the sense of tempered distributions 
\[  D_A K = \int d^n \bx d^n \ovl{\bx} \, (D_A(E(A,\bx)) M(\bx). \]
\end{cor}
\begin{proof}
Note that
\[ (D_A(E(A,\bx))M(\bx) =P(\bx)E(A,\bx) M(\bx) \]
for some polynomial $P(\bx, \ovl{\bx})$ in $\bx, \ovl{\bx}$. This expression has a good Fourier transform in the sense of tempered distributions. Thus, it is enough to set $\by=0$ in Proposition \ref{prop:diffoponfourier}.
\end{proof}
\begin{prop} \label{prop:integrationbyparts}
Let $P(\bx, \ovl{\bx})$ be a polynomial in $\bx,\ovl{\bx}$. Then, for any $i$ we have 
\[ \int d^n \bx  d^n \ovl{\bx} \, \partial_i \left[ P(\bx, \ovl{\bx}) e^{(\bx,A\bx)-\ovl{(\bx,A\bx)}} M(\bx) \right] = 0.  \]
In particular, we can evaluate integrals of the form 
\[ \int d^n \bx d^n \ovl{\bx} \, P(\bx, \ovl{\bx}) e^{(\bx,A\bx)-\ovl{(\bx,A\bx)}} M(\bx) \]
by parts without boundary terms.
\end{prop}
\begin{proof}
We wish to compute these integrals in the sense of tempered distributions. Note that $g(\by)$ is smooth in $\by$, so we find that
\[ \int d^n \bx  d^n \ovl{\bx} \,  P(\bx, \ovl{\bx}) e^{(\bx,A\bx)-\ovl{(\bx,A\bx)}} M(\bx) = P( \p_{\by}, \p_{\ovl{\by}} ) g(\by),  \]
and the right hand side is a smooth function of $\by$. We thus find that 
\[ \int d^n \bx  d^n \ovl{\bx} \, \partial_i \left[ P(\bx, \ovl{\bx}) e^{(\bx,A\bx)-\ovl{(\bx,A\bx)}} M(\bx) \right] = y_i P( \p_{\by}, \p_{\ovl{\by}} ) g(\by) \bigg|_{\by = 0} = 0 .  \]
We are done.
\end{proof}

In summary, we have shown that we may sensibly define the following integral in the sense of tempered distributions:
\[ \int d^m \bx d^m \bar{\bx} \left[ \prod_{i = 1}^m \abs{x_i}^s \right] e^{(\bx,A\bx)-\ovl{(\bx,A \bx)}} . \]
We have then shown that we may safely differentiate under the integral sign and integrate by parts as we have done in Sec. \ref{subsec:twist}, justifying all of the algebraic manipulations in our proof of the twisted Gaiotto formula.

\bibliographystyle{abbrv}
\bibliography{bib}

\end{document}